\documentclass[12pt,a4paper]{article}
\usepackage[T2A]{fontenc}
\usepackage[english]{babel}
\usepackage[utf8]{inputenc}

\usepackage{indentfirst}

\usepackage{orcidlink}

\usepackage[titles]{tocloft}

\usepackage{amsmath}
\usepackage{amsthm}
\usepackage{amsfonts}
\usepackage{amssymb}
\usepackage{MnSymbol,pifont,stmaryrd}

\usepackage{csquotes}
\usepackage[backend=biber,bibencoding=utf8, style=gost-numeric]{biblatex}
\usepackage{hyperref}

\usepackage{tikz}
\usepackage{tkz-euclide}
 \usepackage[RGB]{xcolor}
    \definecolor{myred}{RGB}{255,46,46} 
    \definecolor{myblue}{RGB}{74,210,255} 
    \definecolor{mygreen}{RGB}{0,175,0} 
    \definecolor{mypurple}{RGB}{106,90,205} 
    \definecolor{mysand}{RGB}{253,216,53} 

\theoremstyle{plain}
\newtheorem{theo}{Theorem}[section]
\newtheorem*{theo*}{Theorem}
\newtheorem{lemm}[theo]{Lemma}
\newtheorem{corr}[theo]{Corollary}
\newtheorem{prop}[theo]{Proposition}

\theoremstyle{definition}
\newtheorem{deff}[theo]{Definition}

\newtheorem{nota}[theo]{Notation}

\newtheorem{exam}[theo]{Example}
\newtheorem{rema}[theo]{Remark}
\newtheorem{ques}[theo]{Question}

\newcommand{\la}{\langle}
\newcommand{\ra}{\rangle}
\newcommand{\df}{\quad{\colon}{\longleftrightarrow}\quad}

\newcommand{\tur}{\mathbb{T}}
\newcommand{\ST}{\mathsf{ST}}
\newcommand{\SV}{\mathsf{SV}}
\newcommand{\SE}{\mathsf{SE}}
\newcommand{\lr}{\mathsf{LR}}
\newcommand{\Sp}{\mathsf{Sp}}
\renewcommand{\d}{\mathsf{d^+}\hspace{-0.5pt}}
\renewcommand{\le}{\leqslant}
\renewcommand{\leq}{\leqslant}
\newcommand{\bw}{\mathsf{bw}}

\begin{document}
\title{Axiomatic and Erd\H{o}s--Moon approaches \\ to tournament rankings}

\author{Sergei Nokhrin\footnote{Krasovskii Institute of Mathematics and Mechanics of UB RAS, 620108, Yekaterinburg, Russia; \textit{e-mail}\textup{:} varyag2@mail.ru}, Mikhail Patrakeev\footnote{Krasovskii Institute of Mathematics and Mechanics of UB RAS, 620108, Yekaterinburg, Russia; \textit{e-mail}\textup{:} p17533@gmail.com; \href{https://orcid.org/0000-0001-7654-5208}{orcid.org/0000-0001-7654-5208}}\hspace{2mm}\orcidlink{0000-0001-7654-5208} 
}

\maketitle

\section*{Abstract}

\vspace{4mm}

\vspace{4mm}

Tournament \textit{ranking} is a function that assigns each vertex of a tournament (i.e., a directed graph without loops, in which each pair of different vertexes is connected by exactly one arc) a number called the  \textit{rank} of the vertex. One of approaches to constructing tournament rankings suggests choosing a ranking that satisfies a fixed set of axioms. In another approach, proposed by Erd\H{o}s and Moon, only injective rankings are considered, and among them, one that minimises the number of \textit{backward} arcs is selected (an arc ${x}\,{\to}\,{y}$ is called backward iff the rank of $x$ is less than the rank of $y$). We combine these two approaches as follows: among the rankings that satisfy a fixed set of axioms, we choose one that minimises the number of backward arcs. 

The Erd\H{o}s–Moon approach naturally leads to the question of how small the proportion of backward arcs can be guaranteed when using injective rankings. Erd\H{o}s and Moon showed that the answer to this question is 1/2. A similar question arises in our approach: how small the proportion of backward arcs can be guaranteed when using rankings that satisfy a set of axioms $\mathcal{A}$? We call this number the \textit{Erd\H{o}s-Moon number} of $\mathcal{A}$. We prove that the Erd\H{o}s–Moon number of the \textit{Copeland} axiom equals 3/4.
\vspace{4mm}

\textbf{Keywords:} The linear ordering problem, paired comparisons analysis, tournament ranking, Er\-d\H{o}s–Mo\-on number, Copeland fair ranking, spectral fair ranking, linear fair ranking

\section{Introduction}

How can we choose the ``best'' from a set of objects if about every two different objects of this set we know which of them is ``better'’? For example, the objects could be players in a single round-robin sports tournament with no draws, and then we need to find a way to select the winners of such a tournament. One standard method is to rank the objects (i.e., assign each object a value called its \textit{rank}) and then take the objects with the highest ranks. For example, the number of wins in a tournament is often used as the rank of players.

Usually, the information that ${x}$ is better than ${y}$ is encoded by the presence of the arc ${x}\,{\to}\,{y}$. Thus, all information is represented in the form of a finite tournament (i.e., a directed graph without loops, in which each pair of different vertexes is connected by exactly one arc), whose vertexes represent the objects of our interest. In the sports tournament example, the vertexes are the players, and the arc ${x}\,{\to}\,{y}$ corresponds to the match between ${x}$ and ${y}$, in which ${x}$ won. A tournament \textit{ranking} is a function from the set of vertexes to an arbitrary totally ordered set.

There are various approaches to constructing a ranking of a directed graph. One approach is as follows. 
First, a certain (task-dependent) set of conditions is fixed, which the ranking must satisfy; we call such conditions \textit{fairness axioms}. Then, an arbitrary ranking that satisfies this set of conditions is taken. (For example, in terms of a sports tournament, it is reasonable to require that if player ${x}$ defeated all the players that player ${y}$ defeated, and also defeated someone else, then the rank of ${x}$ should be higher than the rank of ${y}$ --- we call such a condition the \textit{weak} fairness axiom.) This approach is presented, for example, in~\cite{Tennenholtz} and~\cite{Rubinstein}. In~\cite{gonzalez2014paired}, a large number of different fairness axioms are presented and it is shown which of the known ranking methods satisfy each of these axioms. If $\mathcal{A}$ is a set of fairness axioms, then an \textit{$\mathcal{A}$-fair} ranking is a ranking that satisfies the axioms from $\mathcal{A}$.

Another approach was proposed in an article by Erd\H{o}s and Moon~\cite{erdos1965sets}. Their paper considers only injective rankings, i.e., those in which different vertexes have different ranks\footnote{In fact, Erd\H{o}s and Moon construct not a ranking, but a linear order on a set of vertexes. However, a linear order naturally corresponds to an injective ranking.}. Following this approach, among all injective rankings of the tournament, we need to choose one that minimises the number of \textit{backward} arcs (an arc ${x}\,{\to}\,{y}$ is called backward iff the rank of ${x}$ is less than the rank of ${y}$)\footnote{The number of backward arcs is a special case of the Kendall distance~\cite{kendall1990rank}}. In the case of a single round-robin sports tournament, a backward arc corresponds to a match in which the player with the lower rank won. Thus, each backward arc in a sense indicates a flaw in the ranking, and therefore it is desirable to minimise the number of backward arcs. Finding an injective ranking with a minimum number of backward arcs is an NP-complete problem~\cite{charon2010updated}. The problem of searching for good algorithms that allow to obtain injective rankings with a small number of backward arcs is studied within the framework of the linear ordering problem~\cite{marti2011linear}.

It seems natural to try to combine the axiomatic approach with the Erd\H{o}s–Moon approach. That is, among all injective $\mathcal{A}$-fair rankings, choose one that minimises the number of backward arcs. Unfortunately, this method of constructing a ranking is not applicable to some tournaments. For example, for virtually every known set $\mathcal{A}$ of axioms, the tournament on three vertexes ${x}$, ${y}$, and ${z}$  with arcs ${x}\,{\to}\,{y}$, ${y}\,{\to}\,{z}$, and ${z}\,{\to}\,{x}$ has no injective $\mathcal{A}$-fair ranking due to symmetry.

In the Erd\H{o}s–Moon approach, the injectivity condition is essential: without it, the optimal ranking with respect to the number of backward arcs would be one that assigns the same rank to all vertexes. However, such a ranking does not bring us any closer to solving the original problem of selecting the best objects. That is, the injectivity condition is needed, in particular, to weed out obviously bad rankings. But fairness axioms also weed out bad rankings, so in their presence, the injectivity condition can be dispensed with, and then the two approaches under discussion can be easily combined.

\medskip
We propose combining the Erd\H{o}s–Moon approach with the axiomatic approach as follows: \textit{To obtain a tournament ranking, one should choose, among all rankings that satisfy a fixed set of fairness axioms, one that minimises the number of backward arcs.}
\medskip

In this paper, we consider several fairness axioms. In terms of single round-robin sports tournament, these axioms can be described as follows. The out-degree of a player (i.e., a vertex) ${x}$ is the number of matches in which ${x}$ won. The \textit{strict Copeland} fairness axiom\footnote{The tournament ranking defined by the out-degrees of its vertexes was studied by Copeland~\cite{copeland1951reasonable}.}, Definition~\ref{deff.ax.cop}, postulates that if the out-degree of ${x}$ is greater than the out-degree of ${y}$, then the rank of ${x}$ is greater than the rank of ${y}$. The \textit{Copeland} fairness axiom adds to the previous axiom the condition that if the out-degrees of the players are equal, then their ranks are also equal. The \textit{weak} fairness axiom, Definition~\ref{deff.ax.weak}, has already been described above. The \textit{linear} fairness axiom, Definition~\ref{deff.ax.lin}, states that a player's rank should be higher the greater the sum of the ranks of those he has defeated. The idea behind the \textit{spectral} fairness axiom is that the list of ranks of those defeated by a particular player can be viewed as the ``spectrum'' of that player. A partial order relation can be defined on the set of spectra as follows: if there exists a non-rank-decreasing injection from spectrum ${S}_1$ to spectrum ${S}_2$, then ${S}_1\leqslant{S}_2$. The spectral fairness axiom, Definition~\ref{deff.ax.spec}, postulates that the rank of a player must be higher the larger his spectrum.

For our approach to be applicable to all (finite) tournaments, each tournament must have a ranking that satisfies the chosen set of axioms. In the case of the weak fairness axiom, as well as in the case of both Copeland fairness axioms, each tournament has a corresponding fair ranking: such a ranking is the function that maps a vertex to its out-degree. In~\cite{Tennenholtz}, it is proven that each tournament has a spectral fair ranking. We prove, see Theorem~\ref{teor.linear_fair}, that every tournament has a linear fair ranking (and therefore a spectral fair one, because spectral fairness follows from linear fairness).

To prove Theorem~\ref{teor.linear_fair}, we developed a method that allows us to find an $\mathcal{A}$-fair ranking of a tournament ${T}$ for certain sets $\mathcal{A}$ of fairness axioms. The idea behind this method is as follows. Let $\mathsf{SMP}({T})$ denote the set of \textit{simplicial} rankings of tournament ${T}$, i.e., rankings in which the ranks of the vertexes belong to the closed interval $[0,1]$ and their sum is equal to 1. We call a mapping from $\mathsf{SMP}({T})$ to $\mathsf{SMP}({T})$ a \textit{recalculation} after tournament ${T}$. The set $\mathsf{SMP}({T})$ is a simplex, so every continuous recalculation has a fixed point. We define the $\mathcal{A}$-fairness condition for recalculations in such a way that if a recalculation is $\mathcal{A}$-fair, then each of its fixed points is an $\mathcal{A}$-fair ranking of the given tournament, see Section~\ref{sect.fixed.point}.

The method we propose for combining the two approaches is a modification of the Erd\H{o}s–Moon approach, in which the injectivity condition is replaced by the $\mathcal{A}$-fairness condition. The Erd\H{o}s–Moon approach raises the natural question of how small the proportion of backward arcs can be guaranteed when using injective rankings. Erd\H{o}s and Moon showed~\cite{erdos1965sets} that the answer to this question is 1/2. For each set of axioms $\mathcal{A}$, a similar question arises within our approach: how small the proportion of backward arcs can be guaranteed when using $\mathcal{A}$-fair rankings? Formally, these questions can be formulated as follows:

Let $\bw({r},{T})$ be the proportion of backward arcs in ranking ${r}$ of tournament ${T}$, i.e., the ratio of the number of backward arcs to the number of all arcs. Let $\mathcal{A}({T})$ denote the class of those rankings of tournament ${T}$ that satisfy the set of axioms~$\mathcal{A}$. Let $\tur$ denote the class of finite tournaments. We are interested in the real number
\[
\mathsf{EMN}(\mathcal{A})\coloneq
\sup\limits_{{T}\in\tur}\min\limits_{{r}\in\mathcal{A}({T})}\bw({r},{T}).
\]
We call the number $\mathsf{EMN}(\mathcal{A})$ the \textit{Erd\H{o}s–Moon number} of~$\mathcal{A}$. Using these notations, the above-mentioned result of Erd\H{o}s and Moon can be formulated as follows:
\[
\sup\limits_{{T}\in\tur}\min\limits_{{r}\in\mathsf{Inj}({T})}\bw({r},{T})=1/2,
\]
where $\mathsf{Inj}({T})$ is the class of injective rankings of tournament ${T}$. In other words, the Erd\H{o}s–Moon number of the \textit{injectivity} axiom (which states that the ranks of different vertexes are different) is equal to 1/2.

We have found the Erd\H{o}s–Moon number of the strict Copeland axiom and that of the Copeland axiom, see Theorem~\ref{teor.EN(Cop,sCop)=3/4}; both of these numbers are equal to 3/4. The Erd\H{o}s–Moon numbers of the weak fairness axiom, the spectral fairness axiom, and the linear fairness axiom are currently unknown.

\section{Notation, definitions and results}\label{sect.notions.defs.results}

Recall that a \emph{tournament} $\la{V},{E}\ra$ is a directed graph (which means that ${E}\subseteq{V}{\times}{V}\setminus\{\la{x},{x}\ra:{x}\in{V}\}$) such that 
$\la{x},{y}\ra\in{E}\leftrightarrow\la{y},{x}\ra\nin{E}$ for all vertexes ${x}\neq{y}\in{V}$; we assume that all tournaments in this paper are finite. We denote by $\tur$ the class of finite tournaments. 
A \emph{ranking} of a tournament $\la{V},{E}\ra$ is a function from ${V}$ to some totally ordered set.

\begin{nota} 
  Let ${T}=\la{V},{E}\ra$ be a tournament, ${x},{y}\in{V}$, and ${r}$ is a ranking of ${T}$.
  \begin{itemize}
    \item [\ding{46}\,] ${x}\,{\to}\,{y}\ $ means that $\la{x},{y}\ra\in{E}$;
    \item [\ding{46}\,] ${x}^{+}={x}^{+}_{T}\coloneq\{{y}\in{V}:{x}\,{\to}\,{y}\}$;
        \item [\ding{46}\,] the \textit{out-degree} of vertex ${x}$, $\d({x})=\d({x},{T})$, is the cardinality of the set ${x}^{+}$;
    \item [\ding{46}\,] an arc $\la{x},{y}\ra$ is a \emph{backward} arc iff ${r}({x})<{r}({y})$;
    \item [\ding{46}\,] $\mathsf{E}({T})\coloneq$ the set of arcs of tournament ${T}$;
    \item [\ding{46}\,] $\overleftarrow{\mathsf{E}}({r},{T})\coloneq$ the set of backward arcs in ranking ${r}$ of tournament ${T}$;
    \item [\ding{46}\,] $\displaystyle\bw({r},{T})\coloneq\frac{|\overleftarrow{\mathsf{E}}({r},{T})|}{|\mathsf{E}({T})|}$ is the proportion of backward arcs in ranking ${r}$ of tournament ${T}$.
  \end{itemize}
\end{nota}

\begin{deff}\label{deff.EN}
 Suppose that for every tournament ${T}$, $\mathcal{A}({T})$ is a nonempty class of rankings of ${T}$.
 Then the \emph{Erd\H{o}s-Moon number} of $\mathcal{A}$ is defined as follows:
\[
\mathsf{EMN}(\mathcal{A})\coloneq
\sup\limits_{{T}\in\tur}\min\limits_{{r}\in\mathcal{A}({T})}\bw({r},{T}).
\]
\end{deff}
\noindent 
Note that the minimum of the set  
$\{\bw({r},{T}):{r}\in\mathcal{A}({T})\}$ is attained because $\bw({r},{T})$ is a fraction whose numerator a natural number and whose denominator is a positive constant.

\begin{deff}\label{deff.ax.cop}
  Let ${r}$ be a ranking of a tournament ${T}=\la{V},{E}\ra$. We say that ${r}$ is 
  \begin{itemize}
    \item [\ding{46}\,] \emph{non-strict Copeland fair} iff $\d({x})\leq\d({y})$ implies ${r}({x})\leq{r}({y})$
    for all ${x},{y}\in{V}$;
    \item [\ding{46}\,] \emph{strict Copeland fair} iff $\d({x})<\d({y})$ implies ${r}({x})<{r}({y})$
    for all ${x},{y}\in{V}$;
    \item [\ding{46}\,] \emph{Copeland fair} iff ${r}$ is both non-strict and strict Copeland fair;
    \item [\ding{46}\,] $\mathsf{nsCop}({T})$, $\mathsf{sCop}({T})$, and $\mathsf{Cop}({T})$ are the classes of non-strict Copeland fair, strict Copeland fair, and Copeland fair rankings of ${T}$, respectively.
\end{itemize}
\end{deff}

The following two examples show that neither strict nor non-strict Copeland fairness implies the other:

\begin{exam}
Consider an arbitrary tournament such that $\d({a})\neq\d({b})$ for some vertexes ${a},{b}$ and its ranking ${r}$ such that ${r}({x})={r}({y})$ for all vertexes ${x},{y}$. Such a tournament is  non-strict Copeland fair and is not strict Copeland fair.
\end{exam}

\begin{exam}
Consider an arbitrary tournament such that $\d({x})=\d({y})$ for all vertexes ${x},{y}$ and its ranking ${r}$ such that ${r}({a})\neq{r}({b})$ for some vertexes ${a},{b}$. Such a tournament is strict Copeland fair and is not non-strict Copeland fair.
\end{exam}

The classes $\mathsf{nsCop}({T})$, $\mathsf{sCop}({T})$, and $\mathsf{Cop}({T})$ are nonempty for every tournament ${T}$  (consider ${r}\coloneq\mathsf{d}^{+}$), so the Erd\H{o}s-Moon numbers of these classes are correctly defined. In Section~\ref{section.copeland} we prove that $\mathsf{EMN}(\mathsf{Cop})=\mathsf{EMN}(\mathsf{sCop})=3/4$, see Theorem~\ref{teor.EN(Cop,sCop)=3/4}. Also we have $\mathsf{EMN}(\mathsf{nsCop})=0$ because the trivial ranking ${r}\equiv{0}$ is non-strict Copeland fair.

\begin{deff}\label{deff.ax.weak}
  Let ${r}$ be a ranking of a tournament $\la{V},{E}\ra$. We say that ${r}$ is
  \begin{itemize}
    \item [\ding{46}\,] \emph{weakly fair} \  iff \   ${x}^{+}\subset{y}^{+}$ implies ${r}({x})<{r}({y})$
    for all ${x},{y}\in{V}$;
    \item [\ding{46}\,] $\mathsf{Weak}({T})$ is the class of weak fair rankings of ${T}$.
\end{itemize}
\end{deff}
\noindent 
Note that ${x}^{+}\subseteq{y}^{+}$ implies ${x}^{+}\subset{y}^{+}$ because ${y}\,{\to}\,{x}$ in this case.

\begin{rema}
    $\mathsf{Cop}({T})\subseteq\mathsf{sCop}({T})\subseteq\mathsf{Weak}({T})$ for every tournament ${T}$.\hfill\qed
\end{rema}

It follows that $\mathsf{Weak}({T})\neq\varnothing$ for every tournament ${T}$, so the Erd\H{o}s-Moon number of the class $\mathsf{Weak}$ is correctly defined. 

\begin{ques}\label{quest.weak}
    Find the value of $\mathsf{EMN}(\mathsf{Weak})$.
\end{ques}

\begin{deff}
Let ${r}$ be a ranking of a tournament $T=\la{V},{E}\ra$ and  ${x},\,{y}\in {V}$. Then:
\begin{itemize}
    \item [\ding{46}\,] ${x}\leqslant_{r}{y}\df$ there exists an injection ${f}\colon{x}^{+}\to {y}^{+}$ such that ${r}({z})\le {r}\big({f(z)}\big)$   for all ${z}\in {x}^{+}$;
    \item [\ding{46}\,] ${x}<_{r}{y}\df{x}\leqslant_{r}{y}$ and  ${y}\nleqslant_{r}{x}$.
\end{itemize}
\end{deff}
\noindent 
Note that $\leqslant_{r}$ is a preorder (that is, a reflexive transitive relation) on the set ${V}$.

\begin{deff}\label{deff.ax.spec}
  Let ${r}$ be a ranking of tournament $\la{V},{E}\ra$. We say that ${r}$ is
  \begin{itemize}
    \item [\ding{46}\,] \emph{non-strict spectral fair} \  iff \   ${x}\leqslant_{r}{y}$ implies ${r}({x})\leq{r}({y})$
    for all ${x},{y}\in{V}$;
    \item [\ding{46}\,] \emph{strict spectral fair} \   iff \   ${x}<_{r}{y}$ implies ${r}({x})<{r}({y})$
    for all ${x},{y}\in{V}$;
    \item [\ding{46}\,] \emph{spectral fair} \   iff \   ${r}$ is both non-strict and strict spectral fair;
    \item [\ding{46}\,] $\mathsf{Spec}({T})$ is the class of spectral fair rankings of ${T}$.
\end{itemize}
\end{deff}

\begin{deff}\label{deff.ax.lin}
  A {\it positive} ranking of a tournament $\la{V},{E}\ra$ is a function from ${V}$ to the set of positive reals with the natural order.   
  We say that a positive ranking ${r}$ is 
  \begin{itemize}
    \item [\ding{46}\,] \emph{non-strict linear fair}\quad iff\quad 
    $\sum\limits_{{z}\in{x}^+}{r}({z})\leq\sum\limits_{{z}\in{y}^+}{r}({z})$ 
    implies ${r}({x})\leq{r}({y})$
    for all ${x},{y}\in{V}$;
    \item [\ding{46}\,] \emph{strict linear fair}\quad iff\quad 
    $\sum\limits_{{z}\in{x}^+}{r}({z})<\sum\limits_{{z}\in{y}^+}{r}({z})$ 
    implies ${r}({x})<{r}({y})$
    for all ${x},{y}\in{V}$;
    \item [\ding{46}\,] \emph{linear fair}\quad iff\quad ${r}$ is both non-strict and strict linear fair;
    \item [\ding{46}\,] $\mathsf{Lin}({T})$ is the class of linear fair rankings of ${T}$.
\end{itemize}
\end{deff}

Note that if we replace the positiveness condition in the definition of linear fair ranking with the non-negativeness condition, then linear fairness would not imply spectral fairness (consider a tournament with two vertexes whose ranks are~0).

\begin{prop}\label{prop.Lin.Spec.Weak}
    $\mathsf{Lin}({T})\subseteq\mathsf{Spec}({T})\subseteq\mathsf{Weak}({T})$ for every tournament ${T}$.\hfill\qed
\end{prop}

In Section~\ref{sect.fixed.point} we show that $\mathsf{Lin}({T})\neq\varnothing$ for every tournament ${T}$, see Theorem~\ref{teor.linear_fair}, so the Erd\H{o}s-Moon numbers of the classes $\mathsf{Lin}$ and $\mathsf{Spec}$ are also correctly defined. Therefore we may ask:

\begin{ques}\label{quest.Lin.Spec}
    Find the values of $\mathsf{EMN}(\mathsf{Lin})$ and $\mathsf{EMN}(\mathsf{Spec})$.
\end{ques}

\section{A fixed point of a fair recalculation}\label{sect.fixed.point}

In this section we prove that every tournament possesses a linear fair ranking and a spectral fair ranking, see Theorem~\ref{teor.linear_fair}. For this purpose, we have developed a general method that allows us to find a fair ranking for different notions of fairness. We need the following terminology.

\begin{deff}\label{deff.recalculation}
  A {\it simplicial ranking of tournament} $\la{V},{E}\ra$ is a function from ${V}$ to the closed interval $[0,1]$ such that $\sum\limits_{{x}\in{V}}{r}({x})=1$. 
  We denote by $\mathsf{SMP}({T})$ the set of simplicial rankings of tournament ${T}$. 
  A {\it recalculation} after tournament ${T}$ is a function from $\mathsf{SMP}({T})$ to $\mathsf{SMP}({T})$. 
  For a simplicial ranking ${r}$ and a recalculation $\varphi$, we denote the ranking $\varphi({r})$ by ${r}_\varphi$.   
\end{deff}

If we think in terms of a single round-robin sports tournament, then the numbers $\la{r}({x}):{x}\in{V}\ra$ can be viewed as the ratings of players ${x}\in{V}$ before tournament ${T}$, and the numbers $\la{r}_\varphi({x}):{x}\in{V}\ra$ can be viewed as the new ratings of the same players, recalculated based on the results of the tournament. Similar rating recalculation functions are used, for example, in chess. Of particular interest are those functions $\varphi$ for which the player's new rating is higher the more players he defeated and the higher their ratings were.

Our goal is to construct an $\mathcal{A}$-fair tournament ranking. The method in this section is applicable to those axioms $\mathcal{A}$ that can be reformulated in the form 
\begin{equation}
\label{implication}
{P}({r},{x},{y})\Rightarrow{Q}({r},{x},{y})\quad\text{for all }{x},{y}\in{V},    
\end{equation}
where ${P}({r},{x},{y})$ and ${Q}({r},{x},{y})$ are formulas which say something about ranking ${r}$ and vertexes ${x}$ and ${y}$. For example, the strict spectral fairness axiom
\[
{x}<_{r}{y}\Rightarrow{r}({x})<{r}({y})\quad\text{for all }{x},{y}\in{V}
\]
has this form. For an axiom $\mathcal{A}$ that is equivalent to (\ref{implication}), 
we say that a recalculation $\varphi$ after tournament ${T}$ is \textit{$\mathcal{A}$-fair} iff 
\[
{P}({r},{x},{y})\Rightarrow{Q}({r}_\varphi,{x},{y})\quad\text{for all }{r}\in\mathsf{SMP}({T})\text{ and all }{x},{y}\in{V}.
\]
For example, $\varphi$ is \textit{strict spectral fair} iff
\[
{x}<_{r}{y}\Rightarrow{r}_\varphi({x})<{r}_\varphi({y})\quad\text{for all }{r}\in\mathsf{SMP}({T})\text{ and all }{x},{y}\in{V}.
\]

Now, if $\varphi$ is an $\mathcal{A}$-fair recalculation after ${T}$ and a ranking ${r}$ is its fixed point (that is, ${r}_\varphi={r}$), then condition (\ref{implication}) holds, which means that ${r}$ is an $\mathcal{A}$-fair ranking of~${T}$.


\begin{deff}\label{recal_cont}\mbox{}
Let ${T}=\la{V},{E}\ra$ be a tournament. A recalculation ${\varphi}:\,\mathsf{SMP}({T})\,\to\,\mathsf{SMP}({T})$ is \textit{continuous} (\textit{contractive}) iff it is a continuous (contractive) mapping with respect to the metric 
\[{d}({r},{r}^\prime)\coloneq\max\limits_{{x}\in{V}}|{r}({x})-{r}^\prime({x})|.
\]
\end{deff}

If we view a ranking ${r}\in\mathsf{SMP}({T})$ as a point with coordinates $\la{r}({x})\ra_{{x}\in{V}}$, then the set $\mathsf{SMP}({T})$ is a $(|{V}|-1)$-dimensional simplex, so the Brouwer fixed-point theorem and the Banach fixed-point theorem say that 

\begin{rema}\label{fixed_point}\mbox{}
Every continuous (contractive) recalculation has a fixed point.\hfill\qed
\end{rema}

\begin{deff}\label{deff.recal.lin.fair}
  We say that a recalculation $\varphi$ after a tournament ${T}$ is
    \begin{itemize}
    \item [\ding{46}\,] \emph{non-strict linear fair} \ iff 
    
    $\sum\limits_{{z}\in{x}^+}{r}({z})\leq\sum\limits_{{z}\in{y}^+}{r}({z})$ 
    implies ${r}_\varphi({x})\leq{r}_\varphi({y})$
    for all ${r}\in\mathsf{SMP}({T})$ and all ${x},{y}\in{V}$;
    \item [\ding{46}\,] \emph{strict linear fair} \ iff 
    
    $\sum\limits_{{z}\in{x}^+}{r}({z})<\sum\limits_{{z}\in{y}^+}{r}({z})$ 
    implies ${r}_\varphi({x})<{r}_\varphi({y})$
    for all ${r}\in\mathsf{SMP}({T})$ and all ${x},{y}\in{V}$;
    \item [\ding{46}\,] \emph{linear fair} \ iff \ ${r}$ is both non-strict and strict linear fair.
  \end{itemize}
\end{deff}

\begin{lemm}\label{lemm.fixed.point}\mbox{}
If a fixed point of a linear fair recalculation is a positive ranking, then it is a linear fair ranking.
\hfill\qed
\end{lemm}

\begin{theo}\label{teor.linear_fair}\mbox{}
  \begin{itemize}
      \item [1. ] Every tournament possesses a linear fair ranking. 
      \item [2. ] Every tournament possesses a spectral fair ranking. 
  \end{itemize}
\end{theo}

\begin{proof}
Statement (2) follows from (1) and Proposition~\ref{prop.Lin.Spec.Weak}, let us prove (1). Recall that a directed graph is \textit{strongly connected} iff there is a path in each direction between each pair of vertexes of the graph. 

\textit{Case 1.} The tournament ${T}=\la{V},{E}\ra$ is strongly connected. 

Consider a recalculation ${\varphi}$ after tournament ${T}$ such that for every ${r}\in\mathsf{SMP}({T})$ and every ${x}\in{V}$, 
$$
{r}_{\varphi}({x})={\lambda}_{r}^{-1}\cdot\sum\limits_{{z}\in{x}^{+}} {r}({z}),\quad\text{where }\ {\lambda_{r}}=\sum\limits_{{x}\in{V}}\sum\limits_{{z}\in{x}^{+}} {r}({z}).
$$
Note that ${\lambda}_{r}\geqslant 1$ because ${r}\in \mathsf{SMP}({T})$ and ${T}$ is strongly connected. Note also that if ${r}$ is a fixed point of $\varphi$, then  ${r}$ is a positive ranking.

The recalculation ${\varphi}$ is continuous, therefore, by Remark~\ref{fixed_point}, it has a fixed point ${r}$. Also
${\varphi}$ is linear fair and ${r}$ is a positive ranking, so, by Lemma~\ref{lemm.fixed.point}, ${r}$ is a linear fair ranking of ${T}$.

\textit{Case 2.} The tournament ${T}$ is not strongly connected. 

Then ${V}$ can be partitioned into disjoint union of strongly connected components ${V}_{1},\ldots,{V}_{n}$ of ${T}$ in such a way that for all ${i}\neq{j}\in\{1,\ldots,{n}\}$,  ${x}\in{V}_{i}$, and ${y}\in{V}_{j}$, we have
\[
{x}\leftarrow{y}\ \ \Leftrightarrow\ \ {i}<{j}.
\]

Every subtournament $\la{V}_{i},{E}_{i}\ra$ of ${T}$ is strongly connected, so, according to Case 1, it has a linear fair positive ranking ${r}_{i}$. There are positive reals 
${\mu}_1,\,\ldots,\,{\mu}_{n}$ such that if ${x}\in{V}_{i}$, ${y}\in{V}_{j}$, and ${i}<{j}$, then
\[
{\mu}_{i}\cdot{r}_{i}({x})<{\mu}_{j}\cdot{r}_{j}({y}).
\]

It is straightforward to show that the ranking ${r}$ such that ${r}({x})={\mu}_{i}\cdot{r}_{i}(x)$ for ${x}\in{V}_{i}$ is a linear fair positive ranking of ${T}$.
\end{proof}

\begin{rema}\label{ranking_with_equality}\mbox{}
If a tournament $\la{V},{E}\ra$ is strongly connected, then the linear fair positive ranking ${r}$ from the above proof satisfies the equality 
\[
\lambda\cdot{r}({x})=\sum\limits_{{y}\in{x}^{+}}{r}({y})\quad\text{for all }\ {x}\in{V},
\]
where $\lambda\geqslant 1$. This means that vector $\la{r}(x)\ra_{{x}\in{V}}$ is an eigenvector of the tournament matrix (that is, the matrix $\la{m}_{{x},{y}}\ra_{{x},{y}\in{V}}$ such that ${m}_{{x},{y}}=1$ whenever ${y}\in{x}^{+}$ and ${m}_{{x},{y}}=0$ otherwise), all its coordinates are positive reals, and its eigenvalue is also a positive real (that is, this eigenvalue is the Perron–Frobenius eigenvalue).
\end{rema}

\section{The Erd\H{o}s-Moon number of axioms $\mathsf{Cop}$ and $\mathsf{sCop}$}
\label{section.copeland}

In this section, we prove that the Erd\H{o}s–Moon number of the strict Copeland axiom and that of the Copeland axiom are equal to 3/4.

\begin{prop}\label{prop.less}
If  ${r}$ is a strict Copeland fair ranking of a tournament ${T}$, then $\bw({r},{T})<{3}/{4}$.
\end{prop}

\begin{corr}\label{prop.grater}
$\mathsf{EMN}(\mathsf{Cop})\leq\frac{3}{4}\ $ and $\ \mathsf{EMN}(\mathsf{sCop})\leq\frac{3}{4}$.\hfill\qed
\end{corr}

\begin{proof}[Proof of Proposition~\ref{prop.less}.]
Let ${T}=\la{V},{E}\ra$ be a tournament and ${r}$ be its strict Copeland fair ranking. Without loss of generality, we may assume that ${V}=\{1,\ldots,{n}\}$ and that ${i}<{j}$ whenever ${r}(i)<{r}({j})$ for all ${i},{j}\in{V}$. 
Put 
\[
{E}^{<}\coloneq\{\la{i},{j}\ra\in{E}:{i}<{j}\}.
\]

\begin{figure}[h]
\centering

\begin{tikzpicture}
[scale=1,vect/.style={->,shorten >=1pt,>=latex'}]
\tkzSetUpLine[line width=1pt]
\tkzSetUpCircle[color=black, line width=0.2pt]
\tkzSetUpPoint[size=4,circle,fill=black!]
\tkzSetUpArc[delta=0,color=black,line width=.75pt]

\tkzDefPoints{0/0/A, 5/0/B}
\tkzDefPointBy[rotation=center A angle 90](B) \tkzGetPoint{D}
\foreach \i in {1,...,7} {
 \tkzDefPointOnLine[pos={0.125*\i}](A,B)\tkzGetPoint{x\i}
}
\tkzDefPointOnLine[pos=0.125](A,D)\tkzGetPoint{y1}
\tkzDefPointOnLine[pos=0.25](A,D)\tkzGetPoint{y2}
\tkzDefPointOnLine[pos=0.375](A,D)\tkzGetPoint{y3}
\tkzDefPointOnLine[pos=0.5](A,D)\tkzGetPoint{y4}
\tkzDefPointOnLine[pos=0.625](A,D)\tkzGetPoint{y5}
\tkzDefPointOnLine[pos=0.75](A,D)\tkzGetPoint{y6}
\tkzDefPointOnLine[pos=0.875](A,D)\tkzGetPoint{y7}

\tkzDefPointBy[translation= from A to B](y1)\tkzGetPoint{y11}
\tkzDefPointBy[translation= from A to B](y2)\tkzGetPoint{y12}
\tkzDefPointBy[translation= from A to B](y3)\tkzGetPoint{y13}
\tkzDefPointBy[translation= from A to B](y4)\tkzGetPoint{y14}
\tkzDefPointBy[translation= from A to B](y5)\tkzGetPoint{y15}
\tkzDefPointBy[translation= from A to B](y6)\tkzGetPoint{y16}
\tkzDefPointBy[translation= from A to B](y7)\tkzGetPoint{y17}
\tkzDefPointBy[translation= from A to B](D)\tkzGetPoint{C}

\tkzDefPointBy[translation= from A to D](x1)\tkzGetPoint{x11}
\tkzDefPointBy[translation= from A to D](x2)\tkzGetPoint{x12}
\tkzDefPointBy[translation= from A to D](x3)\tkzGetPoint{x13}
\tkzDefPointBy[translation= from A to D](x4)\tkzGetPoint{x14}
\tkzDefPointBy[translation= from A to D](x5)\tkzGetPoint{x15}
\tkzDefPointBy[translation= from A to D](x6)\tkzGetPoint{x16}
\tkzDefPointBy[translation= from A to D](x7)\tkzGetPoint{x17}

\tkzInterLL(x1,x11)(y1,y11)\tkzGetPoint{z11}

\tkzDefPointBy[translation= from A to z11](y1)\tkzGetPoint{z12}
\tkzDefPointBy[translation= from A to z11](z12)\tkzGetPoint{z23}
\tkzDefPointBy[translation= from A to z11](z23)\tkzGetPoint{z34}
\tkzDefPointBy[translation= from A to z11](z34)\tkzGetPoint{z45}
\tkzDefPointBy[translation= from A to z11](z45)\tkzGetPoint{z56}
\tkzDefPointBy[translation= from A to z11](z56)\tkzGetPoint{z67}
\tkzDefPointBy[translation= from A to z11](z67)\tkzGetPoint{z78}
\tkzDefPointBy[translation= from A to z11](x1)\tkzGetPoint{z21}
\tkzDefPointBy[translation= from A to z11](z21)\tkzGetPoint{z32}
\tkzDefPointBy[translation= from A to z11](z32)\tkzGetPoint{z43}
\tkzDefPointBy[translation= from A to z11](z43)\tkzGetPoint{z54}
\tkzDefPointBy[translation= from A to z11](z54)\tkzGetPoint{z65}
\tkzDefPointBy[translation= from A to z11](z65)\tkzGetPoint{z76}
\tkzDefPointBy[translation= from A to z11](z76)\tkzGetPoint{z87}
\tkzDefPointBy[translation= from A to z11](z11)\tkzGetPoint{z22}
\tkzDefPointBy[translation= from A to z11](z22)\tkzGetPoint{z33}
\tkzDefPointBy[translation= from A to z11](z33)\tkzGetPoint{z44}
\tkzDefPointBy[translation= from A to z11](z44)\tkzGetPoint{z55}
\tkzDefPointBy[translation= from A to z11](z55)\tkzGetPoint{z66}
\tkzDefPointBy[translation= from A to z11](z66)\tkzGetPoint{z77}
\tkzDefPointBy[translation= from A to z11](z77)\tkzGetPoint{z88}

\tkzInterLL(x2,x12)(y1,y11)\tkzGetPoint{z21}
\tkzInterLL(x2,x12)(y2,y12)\tkzGetPoint{z22}
\tkzInterLL(x2,x12)(y3,y13)\tkzGetPoint{z23}
\tkzInterLL(x2,x12)(y4,y14)\tkzGetPoint{z24}
\tkzInterLL(x2,x12)(y5,y15)\tkzGetPoint{z25}
\tkzInterLL(x2,x12)(y6,y16)\tkzGetPoint{z26}
\tkzInterLL(x2,x12)(y7,y17)\tkzGetPoint{z27}

\tkzInterLL(y2,y12)(x1,x11)\tkzGetPoint{z12}
\tkzInterLL(y2,y12)(x3,x13)\tkzGetPoint{z32}
\tkzInterLL(y2,y12)(x4,x14)\tkzGetPoint{z42}
\tkzInterLL(y2,y12)(x5,x15)\tkzGetPoint{z52}
\tkzInterLL(y2,y12)(x6,x16)\tkzGetPoint{z62}
\tkzInterLL(y2,y12)(x7,x17)\tkzGetPoint{z72}

\tkzInterLL(x3,x13)(y4,y14)\tkzGetPoint{z34}
\tkzInterLL(x4,x14)(y5,y15)\tkzGetPoint{z45}
\tkzInterLL(x5,x15)(y6,y16)\tkzGetPoint{z56}
\tkzInterLL(x6,x16)(y7,y17)\tkzGetPoint{z67}
\tkzInterLL(x4,x14)(y3,y13)\tkzGetPoint{z43}
\tkzInterLL(x5,x15)(y4,y14)\tkzGetPoint{z54}
\tkzInterLL(x6,x16)(y5,y15)\tkzGetPoint{z65}
\tkzInterLL(x7,x17)(y6,y16)\tkzGetPoint{z76}

\tkzDefPointBy[homothety=center x16 ratio .9](z77)\tkzGetPoint{k1}
\tkzDefPointBy[homothety=center x17 ratio .9](z67)\tkzGetPoint{k2}
\tkzDefPointBy[homothety=center z77 ratio .9](x16)\tkzGetPoint{k3}
\tkzDefPointBy[homothety=center z67 ratio .9](x17)\tkzGetPoint{k4}
\tkzDrawPolygon[green!60!black](k1,k2,k3,k4)

\tkzDefPointBy[translation= from z33 to z22](k1)\tkzGetPoint{k5} 
\tkzDefPointBy[translation= from z33 to z22](k2)\tkzGetPoint{k6} 
\tkzDefPointBy[translation= from z43 to z32](k3)\tkzGetPoint{k7} 
\tkzDefPointBy[translation= from z43 to z32](k4)\tkzGetPoint{k8} 
\tkzDrawPolygon[green!60!black](k5,k6,k7,k8)

\tkzDefPointBy[translation= from z42 to z22](k5)\tkzGetPoint{k9} 
\tkzDefPointBy[translation= from z43 to z23](k6)\tkzGetPoint{k10} 
\tkzDefPointBy[translation= from z43 to z23](k7)\tkzGetPoint{k11} 
\tkzDefPointBy[translation= from z43 to z23](k8)\tkzGetPoint{k12} 
\tkzDrawPolygon[green!60!black](k9,k10,k11,k12)

\tkzDefMidPoint(z22,z33)\tkzGetPoint{m1}
\tkzDefPointOnLine[pos=0.4](m1,z22)\tkzGetPoint{n}
\tkzDefMidPoint(z22,m1)\tkzGetPoint{m2}
\tkzCalcLength(m1,n) \tkzGetLength{ro}
\tkzDefPointBy[translation= from z12 to z32](m1)\tkzGetPoint{m3} \tkzDefCircle[R](m3,\ro) \tkzGetPoint{p} \tkzDrawCircle[fill=blue!50](m3,p)
\tkzDefPointBy[translation= from z12 to z22](m3)\tkzGetPoint{m4} \tkzDefCircle[R](m4,\ro) \tkzGetPoint{p} \tkzDrawCircle[fill=blue!50](m4,p)
\tkzDefPointBy[translation= from z12 to z32](m3)\tkzGetPoint{m5} \tkzDefCircle[R](m5,\ro) \tkzGetPoint{p} \tkzDrawCircle[fill=blue!50](m5,p)
\tkzDefPointBy[translation= from z12 to z32](m4)\tkzGetPoint{m6} \tkzDefCircle[R](m6,\ro) \tkzGetPoint{p} \tkzDrawCircle[fill=blue!50](m6,p)
\tkzDefPointBy[translation= from z21 to z22](m3)\tkzGetPoint{m7} \tkzDefCircle[R](m7,\ro) \tkzGetPoint{p} \tkzDrawCircle[fill=blue!50](m7,p)
\tkzDefPointBy[translation= from z21 to z22](m4)\tkzGetPoint{m8} \tkzDefCircle[R](m8,\ro) \tkzGetPoint{p} \tkzDrawCircle[fill=blue!50](m8,p)
\tkzDefPointBy[translation= from z21 to z22](m6)\tkzGetPoint{m9} \tkzDefCircle[R](m9,\ro) \tkzGetPoint{p} \tkzDrawCircle[fill=blue!50](m9,p)
\tkzDefPointBy[translation= from z21 to z22](m9)\tkzGetPoint{m11} \tkzDefCircle[R](m11,\ro) \tkzGetPoint{p} \tkzDrawCircle[fill=blue!50](m11,p)
\tkzDefPointBy[translation= from z22 to z12](m11)\tkzGetPoint{m10} \tkzDefCircle[R](m10,\ro) \tkzGetPoint{p} \tkzDrawCircle[fill=blue!50](m10,p)
\tkzDefPointBy[translation= from z21 to z22](m11)\tkzGetPoint{m12} \tkzDefCircle[R](m12,\ro) \tkzGetPoint{p} \tkzDrawCircle[fill=blue!50](m12,p)
\tkzDefPointBy[translation= from z22 to z23](m1)\tkzGetPoint{m13} \tkzDefCircle[R](m13,\ro) \tkzGetPoint{p} \tkzDrawCircle[fill=red!50](m13,p)
\tkzDefPointBy[translation= from z52 to z22](m12)\tkzGetPoint{m14} \tkzDefCircle[R](m14,\ro) \tkzGetPoint{p} \tkzDrawCircle[fill=red!50](m14,p)
\tkzDefPointBy[translation= from z12 to z23](m14)\tkzGetPoint{m15} \tkzDefCircle[R](m15,\ro) \tkzGetPoint{p} \tkzDrawCircle[fill=red!50](m15,p)
\tkzDefPointBy[translation= from z32 to z23](m14)\tkzGetPoint{m16} \tkzDefCircle[R](m16,\ro) \tkzGetPoint{p} \tkzDrawCircle[fill=red!50](m16,p)
\tkzDefPointBy[translation= from z12 to z23](m15)\tkzGetPoint{m17} \tkzDefCircle[R](m17,\ro) \tkzGetPoint{p} \tkzDrawCircle[fill=red!50](m17,p)

\tkzDrawSegments(z22,x12 z22,y12  C,x12 C,y12)
\tkzDrawSegments[line width=0.5pt](z23,y13 z24,y14  z26,y16 z27,y17 z32,x13 z42,x14 z62,x16 z72,x17 z52,x15 z25,y15)
\tkzDrawSegments[line width=0.5pt,dashed](z21,y11 z22,x2 z22,y2  z23,y3 z24,y4 z25,y5 z26,y6 z27,y7 z12,x11 z32,x3 z42,x4 z52,x5 z62,x6 z72,x7 x12,D y12,B)

\tkzDrawLines[add=-0.2 and -0.2](z23,z32 x17,y17 z34,z43 z45,z54 z56,z65 z67,z76)
\tkzDrawLines[add=-0.2 and -0.2](z22,z33 z33,z44 z44,z55 z55,z66 z66,z77 z77,C)

\tkzLabelPoint[above right](x2){$3$}
\tkzLabelPoint[above right](x3){$3$}
\tkzLabelPoint[above right](x4){$5$}
\tkzLabelPoint[above right](x5){$5$}
\tkzLabelPoint[above right](x6){$6$}
\tkzLabelPoint[above right](x7){$9$}
\tkzLabelPoint[shift={(2mm,5.9mm)}](B){$r$}

\tkzLabelPoint[below right=+0.5mm](z22){$1$} 
\tkzLabelPoint[below right=+0.5mm](z32){$2$}
\tkzLabelPoint[below right=+0.5mm](z42){$3$}
\tkzLabelPoint[below right=+0.5mm](z52){$4$}
\tkzLabelPoint[below right=+0.5mm](z62){$5$}
\tkzLabelPoint[below right=+0.5mm](z72){$6$}
\tkzLabelPoint[below right=-0.3mm](y12){\small\textit{No.}}

\tkzLabelPoint[shift={(3.2mm,6mm)}](D){$r$} 
\tkzLabelPoint[below right=+0.5mm](y7){$6$} 
\tkzLabelPoint[below right=+0.5mm](y6){$5$} 
\tkzLabelPoint[below right=+0.5mm](y5){$5$} 
\tkzLabelPoint[below right=+0.5mm](y4){$3$} 
\tkzLabelPoint[below right=+0.5mm](y3){$3$} 
\tkzLabelPoint[below right=+0.5mm](D){$9$} 

\tkzLabelPoint[above left=-1mm](x12){\small\textit{No.}}
\tkzLabelPoint[below left=+0.5mm](x12){$6$} 
\tkzLabelPoint[below left=+0.5mm](z27){$5$} 
\tkzLabelPoint[below left=+0.5mm](z26){$4$}
\tkzLabelPoint[below left=+0.5mm](z25){$3$} 
\tkzLabelPoint[below left=+0.5mm](z24){$2$} 
\tkzLabelPoint[below left=+0.5mm](z23){$1$} 
\end{tikzpicture}

\caption{Example of a tournament and its ranking}
\label{fig:tourTable}
\end{figure}
It is convenient to visualise a tournament in the form of a table. For instance, Figure~\ref{fig:tourTable} represents a tournament with vertexes $\{1,\ldots,6\}$, whose ranks are 3, 3, 5, 5, 6, and 9, respectively. The arcs of the tournament correspond to the circles; for example, the arc ${1}\,{\to}\,{2}$ corresponds to the leftmost circle. The out-degree of a vertex equals the number of circles in the column with its number. It is easy to see that the ranking of the tournament is strict Copeland fair, but not Copeland fair (because vertexes $4$ and $5$ have the same out-degree but different ranks). The arcs from the set ${E}^{<}$ are coloured red; the remaining arcs are coloured blue. The backward arcs are outlined with a square green frame.

Note that $\overleftarrow{\mathsf{E}}({r},{T})\subseteq{E}^{<}$,   
so it is sufficient to prove 
$$\frac{|{E}^{<}|}{|{E}|}<\frac{3}{4}.$$ 

\begin{figure}[h]
\centering
\resizebox{6cm}{!}{
\begin{tikzpicture}
[scale=1,vect/.style={->,shorten >=1pt,>=latex'}]
\tkzSetUpLine[line width=1pt]
\tkzSetUpCircle[color=black, line width=1pt]
\tkzSetUpPoint[size=4,circle,fill=black!]
\tkzSetUpArc[delta=0,color=black,line width=.75pt]

\tkzDefPoints{0/0/A, 7/0/B}
\tkzDefPointBy[rotation=center A angle 90](B) \tkzGetPoint{D}
\tkzDefPointOnLine[pos=0.1](A,B)\tkzGetPoint{x1}
\tkzDefPointOnLine[pos=0.2](A,B)\tkzGetPoint{x2}
\tkzDefPointOnLine[pos=0.3](A,B)\tkzGetPoint{x3}
\tkzDefPointOnLine[pos=0.4](A,B)\tkzGetPoint{x4}
\tkzDefPointOnLine[pos=0.5](A,B)\tkzGetPoint{x5}
\tkzDefPointOnLine[pos=0.6](A,B)\tkzGetPoint{x6}
\tkzDefPointOnLine[pos=0.7](A,B)\tkzGetPoint{x7}
\tkzDefPointOnLine[pos=0.8](A,B)\tkzGetPoint{x8}
\tkzDefPointOnLine[pos=0.9](A,B)\tkzGetPoint{x9}
\tkzDefPointOnLine[pos=0.1](A,D)\tkzGetPoint{y1}
\tkzDefPointOnLine[pos=0.2](A,D)\tkzGetPoint{y2}
\tkzDefPointOnLine[pos=0.3](A,D)\tkzGetPoint{y3}
\tkzDefPointOnLine[pos=0.4](A,D)\tkzGetPoint{y4}
\tkzDefPointOnLine[pos=0.5](A,D)\tkzGetPoint{y5}
\tkzDefPointOnLine[pos=0.6](A,D)\tkzGetPoint{y6}
\tkzDefPointOnLine[pos=0.7](A,D)\tkzGetPoint{y7}
\tkzDefPointOnLine[pos=0.8](A,D)\tkzGetPoint{y8}
\tkzDefPointOnLine[pos=0.9](A,D)\tkzGetPoint{y9}
\tkzDefPointBy[translation= from A to B](y1)\tkzGetPoint{y11}
\tkzDefPointBy[translation= from A to B](y2)\tkzGetPoint{y12}
\tkzDefPointBy[translation= from A to B](y3)\tkzGetPoint{y13}
\tkzDefPointBy[translation= from A to B](y4)\tkzGetPoint{y14}
\tkzDefPointBy[translation= from A to B](y5)\tkzGetPoint{y15}
\tkzDefPointBy[translation= from A to B](y6)\tkzGetPoint{y16}
\tkzDefPointBy[translation= from A to B](y7)\tkzGetPoint{y17}
\tkzDefPointBy[translation= from A to B](y8)\tkzGetPoint{y18}
\tkzDefPointBy[translation= from A to B](y9)\tkzGetPoint{y19}
\tkzDefPointBy[translation= from A to B](D)\tkzGetPoint{C}
\tkzDefPointBy[translation= from A to D](x1)\tkzGetPoint{x11}
\tkzDefPointBy[translation= from A to D](x2)\tkzGetPoint{x12}
\tkzDefPointBy[translation= from A to D](x3)\tkzGetPoint{x13}
\tkzDefPointBy[translation= from A to D](x4)\tkzGetPoint{x14}
\tkzDefPointBy[translation= from A to D](x5)\tkzGetPoint{x15}
\tkzDefPointBy[translation= from A to D](x6)\tkzGetPoint{x16}
\tkzDefPointBy[translation= from A to D](x7)\tkzGetPoint{x17}
\tkzDefPointBy[translation= from A to D](x8)\tkzGetPoint{x18}
\tkzDefPointBy[translation= from A to D](x9)\tkzGetPoint{x19}

\tkzInterLL(x5,x15)(y5,y15)\tkzGetPoint{z}
\tkzInterLL(x1,x11)(y1,y11)\tkzGetPoint{z11}

\tkzDefPointBy[translation= from A to z11](y1)\tkzGetPoint{z12}
\tkzDefPointBy[translation= from A to z11](z12)\tkzGetPoint{z23}
\tkzDefPointBy[translation= from A to z11](z23)\tkzGetPoint{z34}
\tkzDefPointBy[translation= from A to z11](z34)\tkzGetPoint{z45}
\tkzDefPointBy[translation= from A to z11](z45)\tkzGetPoint{z56}
\tkzDefPointBy[translation= from A to z11](z56)\tkzGetPoint{z67}
\tkzDefPointBy[translation= from A to z11](z67)\tkzGetPoint{z78}
\tkzDefPointBy[translation= from A to z11](z78)\tkzGetPoint{z89}
\tkzDefPointBy[translation= from A to z11](x1)\tkzGetPoint{z21}
\tkzDefPointBy[translation= from A to z11](z21)\tkzGetPoint{z32}
\tkzDefPointBy[translation= from A to z11](z32)\tkzGetPoint{z43}
\tkzDefPointBy[translation= from A to z11](z43)\tkzGetPoint{z54}
\tkzDefPointBy[translation= from A to z11](z54)\tkzGetPoint{z65}
\tkzDefPointBy[translation= from A to z11](z65)\tkzGetPoint{z76}
\tkzDefPointBy[translation= from A to z11](z76)\tkzGetPoint{z87}
\tkzDefPointBy[translation= from A to z11](z87)\tkzGetPoint{z98}
\tkzDefPointBy[translation= from A to z11](z11)\tkzGetPoint{z22}
\tkzDefPointBy[translation= from A to z11](z22)\tkzGetPoint{z33}
\tkzDefPointBy[translation= from A to z11](z33)\tkzGetPoint{z44}
\tkzDefPointBy[translation= from A to z11](z44)\tkzGetPoint{z55}
\tkzDefPointBy[translation= from A to z11](z55)\tkzGetPoint{z66}
\tkzDefPointBy[translation= from A to z11](z66)\tkzGetPoint{z77}
\tkzDefPointBy[translation= from A to z11](z77)\tkzGetPoint{z88}
\tkzDefPointBy[translation= from A to z11](z88)\tkzGetPoint{z99}

\tkzInterLL(x2,x12)(y8,y18)\tkzGetPoint{z28}
\tkzInterLL(x7,x17)(y8,y18)\tkzGetPoint{z78}
\tkzInterLL(x2,x12)(y3,y13)\tkzGetPoint{z23}
\tkzInterLL(x7,x17)(y3,y13)\tkzGetPoint{z73}

\tkzDrawPolygon[fill=myblue](A,x5,z,y5)
\tkzDrawPolygon[fill=myred](B,x5,z,y15)
\tkzDrawPolygon[fill=myblue](C,x15,z,y15)
\tkzDrawPolygon[fill=mygreen](D,x15,z,y5)

\tkzDrawLine[add=0 and 0.1,-Stealth](A, B)
\tkzDrawLine[add=0 and 0.1,-Stealth](A, D)

\tkzDrawSegment[line width=0.5pt](y1, y11)
\tkzDrawSegment[line width=0.5pt](y2, y12)
\tkzDrawSegment[line width=0.5pt](y3, y13)
\tkzDrawSegment[line width=0.5pt](y4, y14)
\tkzDrawSegment(y5, y15)
\tkzDrawSegment[line width=0.5pt](y6, y16)
\tkzDrawSegment[line width=0.5pt](y7, y17)
\tkzDrawSegment[line width=0.5pt](y8, y18)
\tkzDrawSegment[line width=0.5pt](y9, y19)
\tkzDrawSegment(D, C)
\tkzDrawSegment[line width=0.5pt](x1, x11)
\tkzDrawSegment[line width=0.5pt](x2, x12)
\tkzDrawSegment[line width=0.5pt](x3, x13)
\tkzDrawSegment[line width=0.5pt](x4, x14)
\tkzDrawSegment(x5, x15)
\tkzDrawSegment[line width=0.5pt](x6, x16)
\tkzDrawSegment[line width=0.5pt](x7, x17)
\tkzDrawSegment[line width=0.5pt](x8, x18)
\tkzDrawSegment[line width=0.5pt](x9, x19)
\tkzDrawSegment(B, C)

\tkzDrawLine[add=-0.2 and -0.2](x1,y1)
\tkzDrawLine[add=-0.2 and -0.2](A,z11)

\tkzDrawLine[add=-0.2 and -0.2](z12,z21)
\tkzDrawLine[add=-0.2 and -0.2](z11,z22)

\tkzDrawLine[add=-0.2 and -0.2](z23,z32)
\tkzDrawLine[add=-0.2 and -0.2](z22,z33)

\tkzDrawLine[add=-0.2 and -0.2](z34,z43)
\tkzDrawLine[add=-0.2 and -0.2](z33,z44)

\tkzDrawLine[add=-0.2 and -0.2](z45,z54)
\tkzDrawLine[add=-0.2 and -0.2](z44,z55)

\tkzDrawLine[add=-0.2 and -0.2](x19,y19)
\tkzDrawLine[add=-0.2 and -0.2](z99,C)

\tkzDrawLine[add=-0.2 and -0.2](z56,z65)
\tkzDrawLine[add=-0.2 and -0.2](z55,z66)

\tkzDrawLine[add=-0.2 and -0.2](z67,z76)
\tkzDrawLine[add=-0.2 and -0.2](z66,z77)

\tkzDrawLine[add=-0.2 and -0.2](z78,z87)
\tkzDrawLine[add=-0.2 and -0.2](z77,z88)

\tkzDrawLine[add=-0.2 and -0.2](z89,z98)
\tkzDrawLine[add=-0.2 and -0.2](z88,z99)

\tkzLabelLine[pos=1.08,right=-0.8mm](A,D){$j$}
\tkzLabelLine[pos=0.55,left](A,D){$l\!+\!1$}
\tkzLabelLine[pos=0.45,left](A,D){$l$}
\tkzLabelLine[pos=0.05,left](A,D){$1$}
\tkzLabelLine[pos=0.95,left](A,D){$2l$}
\tkzLabelPoint[above left](z28){{\Huge\textbf{A}}}
\tkzLabelPoint[above left](z78){{\Huge\textbf{D}}}
\tkzLabelPoint[above left](z23){{\Huge\textbf{C}}}
\tkzLabelPoint[above left](z73){{\Huge\textbf{B}}}
\tkzLabelPoint[below left](x6){$l\!+\!1$}

\tkzLabelLine[pos=1.07,above](A,B){$i$}
\tkzLabelLine[pos=0.9,below](A,x5){$l$}
\tkzLabelLine[pos=0.05,below](A,B){$1$}
\tkzLabelLine[pos=0.95,below](A,B){$2l$}
\end{tikzpicture}
}
\caption{Even case. $\text{A}\leqslant\text{B}$ because $\text{C}=\text{D}$ and $\text{A+C}\leqslant\text{D+B}$}
\label{fig:even:1}
\end{figure}

\emph{Case ${n}=2{l}$}. First we show
\begin{equation}\label{Lemma.3.even'}
|\{\la{i},{j}\ra\in{E}:{i}\leq{l},\:{j}>{l}\}|\ \leq\ |\{\la{i},{j}\ra\in{E}:{j}\leq{l},\:{i}>{l}\}|,  
\end{equation}
--- that is, in terms of Figure~\ref{fig:even:1}, the number of arcs in the green area (denoted by $\text{A}$) does not exceed the number of arcs in the red area (denoted by $\text{B}$).

Indeed, since for each pair of cells symmetrical with respect to the diagonal, exactly one corresponds to an arc, the number of arcs in each of the two blue regions is equal to $\frac{{l}({l}-1)}{2}$, that is
\[
|\{\la{i},{j}\ra\in{E}:{i}\leq{l},\:{j}\leq{l}\}|
\ =\ \textstyle\frac{{l}({l}-1)}{2}\ =\ 
|\{\la{i},{j}\ra\in{E}:{i}>{l},\:{j}>{l}\}|.
\]
By strict Copeland fairness, the number of arcs in the columns does not decrease as the column number increases:
\vspace{-3mm}
\[
|\{\la{i},{j}\ra\in{E}:{i}\leq{l}\}|\ =\ 
\sum_{{i}=1}^{{l}}\d({i})\ \leq \ \sum_{{i}={l}+1}^{2{l}}\d({i})\ =\ 
|\{\la{i},{j}\ra\in{E}:{i}>{l}\}|.
\]
\vspace{-3mm}
Thus, 
\vspace{-3mm}
\begin{align*}
&|\{\la{i},{j}\ra\in{E}:{i}\leq{l},\:{j}>{l}\}|\ =\ \\
= \ &|\{\la{i},{j}\ra\in{E}:{i}\leq{l}\}|\  
-\ |\{\la{i},{j}\ra\in{E}:{i}\leq{l},\:{j}\leq{l}\}|\ \leq\ \\
\ \leq\ &|\{\la{i},{j}\ra\in{E}:{i}>{l}\}|\ -\ |\{\la{i},{j}\ra\in{E}:{i}>{l},\:{j}>{l}\}|\ =\ \\
=\ &|\{\la{i},{j}\ra\in{E}:{i}>{l},\:{j}\leq{l}\}|
\end{align*}
---  that is, in terms of Figure~\ref{fig:even:1}, 
$\text{A}=(\text{A}+\text{C})-\text{C}\leqslant(\text{B}+\text{D})-\text{D}=\text{B}$.

Now, using (\ref{Lemma.3.even'}), we can write
\begin{align*}
&2\,|\{\la{i},{j}\ra\in{E}^{<}:{i}\leq{l},\:{j}>{l}\}|\ = \ 2\,|\{\la{i},{j}\ra\in{E}:{i}\leq{l},\:{j}>{l}\}|\ \leq\\
\ \leq\ &|\{\la{i},{j}\ra\in{E}:{i}\leq{l},\:{j}>{l}\}| + 
|\{\la{i},{j}\ra\in{E}:{j}\leq{l},\:{i}>{l}\}|\ =\\
\ =\ &|\{\la{i},{j}\ra\in{E}:{i}\leq{l},\:{j}>{l}\}| + 
|\{\la{j},{i}\ra\in{E}:{i}\leq{l},\:{j}>{l}\}|\ =\ {l}^2,
\end{align*}
so we have, 
\begin{equation}\label{Lemma.4.even'}\textstyle
    |\{\la{i},{j}\ra\in{E}^{<},\:{i}\leq{l},\:{j}>{l}\}|\ \leq\ \frac{{l}^2}{2}
\end{equation}
---  that is, in terms of Figure~\ref{fig:even:1}, $\text{A}\leq\frac{{l}^2}{2}.$

\begin{figure}[h]
\centering
\resizebox{6cm}{!}{
\begin{tikzpicture}
[scale=1,vect/.style={->,shorten >=1pt,>=latex'}]
\tkzSetUpLine[line width=1pt]
\tkzSetUpCircle[color=black, line width=1pt]
\tkzSetUpPoint[size=4,circle,fill=black!]
\tkzSetUpArc[delta=0,color=black,line width=.75pt]

\tkzDefPoints{0/0/A, 7/0/B}
\tkzDefPointBy[rotation=center A angle 90](B) \tkzGetPoint{D}
\tkzDefPointOnLine[pos=0.1](A,B)\tkzGetPoint{x1}
\tkzDefPointOnLine[pos=0.2](A,B)\tkzGetPoint{x2}
\tkzDefPointOnLine[pos=0.3](A,B)\tkzGetPoint{x3}
\tkzDefPointOnLine[pos=0.4](A,B)\tkzGetPoint{x4}
\tkzDefPointOnLine[pos=0.5](A,B)\tkzGetPoint{x5}
\tkzDefPointOnLine[pos=0.6](A,B)\tkzGetPoint{x6}
\tkzDefPointOnLine[pos=0.7](A,B)\tkzGetPoint{x7}
\tkzDefPointOnLine[pos=0.8](A,B)\tkzGetPoint{x8}
\tkzDefPointOnLine[pos=0.9](A,B)\tkzGetPoint{x9}
\tkzDefPointOnLine[pos=0.1](A,D)\tkzGetPoint{y1}
\tkzDefPointOnLine[pos=0.2](A,D)\tkzGetPoint{y2}
\tkzDefPointOnLine[pos=0.3](A,D)\tkzGetPoint{y3}
\tkzDefPointOnLine[pos=0.4](A,D)\tkzGetPoint{y4}
\tkzDefPointOnLine[pos=0.5](A,D)\tkzGetPoint{y5}
\tkzDefPointOnLine[pos=0.6](A,D)\tkzGetPoint{y6}
\tkzDefPointOnLine[pos=0.7](A,D)\tkzGetPoint{y7}
\tkzDefPointOnLine[pos=0.8](A,D)\tkzGetPoint{y8}
\tkzDefPointOnLine[pos=0.9](A,D)\tkzGetPoint{y9}
\tkzDefPointBy[translation= from A to B](y1)\tkzGetPoint{y11}
\tkzDefPointBy[translation= from A to B](y2)\tkzGetPoint{y12}
\tkzDefPointBy[translation= from A to B](y3)\tkzGetPoint{y13}
\tkzDefPointBy[translation= from A to B](y4)\tkzGetPoint{y14}
\tkzDefPointBy[translation= from A to B](y5)\tkzGetPoint{y15}
\tkzDefPointBy[translation= from A to B](y6)\tkzGetPoint{y16}
\tkzDefPointBy[translation= from A to B](y7)\tkzGetPoint{y17}
\tkzDefPointBy[translation= from A to B](y8)\tkzGetPoint{y18}
\tkzDefPointBy[translation= from A to B](y9)\tkzGetPoint{y19}
\tkzDefPointBy[translation= from A to B](D)\tkzGetPoint{C}
\tkzDefPointBy[translation= from A to D](x1)\tkzGetPoint{x11}
\tkzDefPointBy[translation= from A to D](x2)\tkzGetPoint{x12}
\tkzDefPointBy[translation= from A to D](x3)\tkzGetPoint{x13}
\tkzDefPointBy[translation= from A to D](x4)\tkzGetPoint{x14}
\tkzDefPointBy[translation= from A to D](x5)\tkzGetPoint{x15}
\tkzDefPointBy[translation= from A to D](x6)\tkzGetPoint{x16}
\tkzDefPointBy[translation= from A to D](x7)\tkzGetPoint{x17}
\tkzDefPointBy[translation= from A to D](x8)\tkzGetPoint{x18}
\tkzDefPointBy[translation= from A to D](x9)\tkzGetPoint{x19}

\tkzInterLL(x5,x15)(y5,y15)\tkzGetPoint{z}
\tkzInterLL(x1,x11)(y1,y11)\tkzGetPoint{z11}

\tkzDefPointBy[translation= from A to z11](y1)\tkzGetPoint{z12}
\tkzDefPointBy[translation= from A to z11](z12)\tkzGetPoint{z23}
\tkzDefPointBy[translation= from A to z11](z23)\tkzGetPoint{z34}
\tkzDefPointBy[translation= from A to z11](z34)\tkzGetPoint{z45}
\tkzDefPointBy[translation= from A to z11](z45)\tkzGetPoint{z56}
\tkzDefPointBy[translation= from A to z11](z56)\tkzGetPoint{z67}
\tkzDefPointBy[translation= from A to z11](z67)\tkzGetPoint{z78}
\tkzDefPointBy[translation= from A to z11](z78)\tkzGetPoint{z89}
\tkzDefPointBy[translation= from A to z11](x1)\tkzGetPoint{z21}
\tkzDefPointBy[translation= from A to z11](z21)\tkzGetPoint{z32}
\tkzDefPointBy[translation= from A to z11](z32)\tkzGetPoint{z43}
\tkzDefPointBy[translation= from A to z11](z43)\tkzGetPoint{z54}
\tkzDefPointBy[translation= from A to z11](z54)\tkzGetPoint{z65}
\tkzDefPointBy[translation= from A to z11](z65)\tkzGetPoint{z76}
\tkzDefPointBy[translation= from A to z11](z76)\tkzGetPoint{z87}
\tkzDefPointBy[translation= from A to z11](z87)\tkzGetPoint{z98}
\tkzDefPointBy[translation= from A to z11](z11)\tkzGetPoint{z22}
\tkzDefPointBy[translation= from A to z11](z22)\tkzGetPoint{z33}
\tkzDefPointBy[translation= from A to z11](z33)\tkzGetPoint{z44}
\tkzDefPointBy[translation= from A to z11](z44)\tkzGetPoint{z55}
\tkzDefPointBy[translation= from A to z11](z55)\tkzGetPoint{z66}
\tkzDefPointBy[translation= from A to z11](z66)\tkzGetPoint{z77}
\tkzDefPointBy[translation= from A to z11](z77)\tkzGetPoint{z88}
\tkzDefPointBy[translation= from A to z11](z88)\tkzGetPoint{z99}

\tkzInterLL(x2,x12)(y8,y18)\tkzGetPoint{z28}
\tkzInterLL(x7,x17)(y8,y18)\tkzGetPoint{z78}
\tkzInterLL(x2,x12)(y3,y13)\tkzGetPoint{z23}
\tkzInterLL(x7,x17)(y3,y13)\tkzGetPoint{z73}

\tkzDrawPolygon[fill=myblue!80!black](y1,y5,z45,z44,z34,z33,z23,z22,z12,z11)
\tkzDrawPolygon[fill=myblue!70](x15,x19,z99,z89,z88,z78,z77,z67,z66,z56)
\tkzDrawPolygon[fill=mygreen](D,x15,z,y5)

\tkzDrawLine[add=0 and 0.1,-Stealth](A, B)
\tkzDrawLine[add=0 and 0.1,-Stealth](A, D)

\tkzDrawSegment[line width=0.5pt](y1, y11)
\tkzDrawSegment[line width=0.5pt](y2, y12)
\tkzDrawSegment[line width=0.5pt](y3, y13)
\tkzDrawSegment[line width=0.5pt](y4, y14)
\tkzDrawSegment(y5, y15)
\tkzDrawSegment[line width=0.5pt](y6, y16)
\tkzDrawSegment[line width=0.5pt](y7, y17)
\tkzDrawSegment[line width=0.5pt](y8, y18)
\tkzDrawSegment[line width=0.5pt](y9, y19)
\tkzDrawSegment(D, C)
\tkzDrawSegment[line width=0.5pt](x1, x11)
\tkzDrawSegment[line width=0.5pt](x2, x12)
\tkzDrawSegment[line width=0.5pt](x3, x13)
\tkzDrawSegment[line width=0.5pt](x4, x14)
\tkzDrawSegment(x5, x15)
\tkzDrawSegment[line width=0.5pt](x6, x16)
\tkzDrawSegment[line width=0.5pt](x7, x17)
\tkzDrawSegment[line width=0.5pt](x8, x18)
\tkzDrawSegment[line width=0.5pt](x9, x19)
\tkzDrawSegment(B, C)

\tkzDrawLine[add=-0.2 and -0.2](x1,y1)
\tkzDrawLine[add=-0.2 and -0.2](A,z11)

\tkzDrawLine[add=-0.2 and -0.2](z12,z21)
\tkzDrawLine[add=-0.2 and -0.2](z11,z22)

\tkzDrawLine[add=-0.2 and -0.2](z23,z32)
\tkzDrawLine[add=-0.2 and -0.2](z22,z33)

\tkzDrawLine[add=-0.2 and -0.2](z34,z43)
\tkzDrawLine[add=-0.2 and -0.2](z33,z44)

\tkzDrawLine[add=-0.2 and -0.2](z45,z54)
\tkzDrawLine[add=-0.2 and -0.2](z44,z55)

\tkzDrawLine[add=-0.2 and -0.2](x19,y19)
\tkzDrawLine[add=-0.2 and -0.2](z99,C)

\tkzDrawLine[add=-0.2 and -0.2](z56,z65)
\tkzDrawLine[add=-0.2 and -0.2](z55,z66)

\tkzDrawLine[add=-0.2 and -0.2](z67,z76)
\tkzDrawLine[add=-0.2 and -0.2](z66,z77)

\tkzDrawLine[add=-0.2 and -0.2](z78,z87)
\tkzDrawLine[add=-0.2 and -0.2](z77,z88)

\tkzDrawLine[add=-0.2 and -0.2](z89,z98)
\tkzDrawLine[add=-0.2 and -0.2](z88,z99)

\tkzLabelLine[pos=1.08,right=-0.8mm](A,D){$j$}
\tkzLabelLine[pos=0.55,left](A,D){$l\!+\!1$}
\tkzLabelLine[pos=0.45,left](A,D){$l$}
\tkzLabelLine[pos=0.05,left](A,D){$1$}
\tkzLabelLine[pos=0.95,left](A,D){$2l$}
\tkzLabelPoint[below left](x6){$l\!+\!1$}

\tkzLabelLine[pos=1.07,above](A,B){$i$}
\tkzLabelLine[pos=0.9,below](A,x5){$l$}
\tkzLabelLine[pos=0.05,below](A,B){$1$}
\tkzLabelLine[pos=0.95,below](A,B){$2l$}

\tkzLabelPoint[above left](z28){{\Huge\textbf{A}}}
\tkzLabelPoint[above left](z78){{\Huge\textbf{E}}}
\tkzLabelPoint[above left](z23){{\Huge\textbf{F}}}
\end{tikzpicture}
}
\caption{Even case. $\text{E}+\text{A}+\text{F}\leqslant\textstyle\frac{{l}({l}-1)}{2} + \frac{{l}^2}{2} + \frac{{l}({l}-1)}{2}$}
\label{fig:even:2}
\end{figure}

Next, using (\ref{Lemma.4.even'}), we get (see Figure~\ref{fig:even:2})
\begin{align*}
|{E}^{<}|\!\ =\ &|\{\la{i},{j}\ra\in{E}^{<},\:{i}>{l}\}|\ +\\
+\ &|\{\la{i},{j}\ra\in{E}^{<},\:{i}\leq{l},\:{j}>{l}\}|\ +\\
+\ &|\{\la{i},{j}\ra\in{E}^{<},\:{i}\leq{l},\:{j}\leq{l}\}|\ \leq\\
\leq\ &\textstyle\frac{{l}({l}-1)}{2} + \frac{{l}^2}{2} + \frac{{l}({l}-1)}{2} \ =\ \frac{{l}(3{l}-2)}{2}.
\end{align*} 

It follows that
\[
\frac{|{E}^{<}|}{|{E}|}\ \leq\ \frac{\frac{{l}(3{l}-2)}{2}}{\frac{2{l}(2{l}-1)}{2}}\ =\ 
\frac{3{l}-2}{4{l}-2} 
\ <\  \frac{3}{4}.
\]

\emph{Case ${n}=2{l}+1$}. Similar to the even case, we have
\begin{equation}\label{Lemma.3.odd'}
|\{\la{i},{j}\ra\in{E}:{i}\leq{l},\:{j}>{l}\}|\ \leq\ |\{\la{i},{j}\ra\in{E}:{j}\leq{l}\!+\!1,\:{i}>{l}\!+\!1\}|.  
\end{equation}
\begin{figure}[h]
\centering
\resizebox{6cm}{!}{
\begin{tikzpicture}
[scale=1,vect/.style={->,shorten >=1pt,>=latex'}]
\tkzSetUpLine[line width=1pt]
\tkzSetUpCircle[color=black, line width=1pt]
\tkzSetUpPoint[size=4,circle,fill=black!]
\tkzSetUpArc[delta=0,color=black,line width=.75pt]

\tkzDefPoints{0/0/A, 7/0/B}
\tkzDefPointBy[rotation=center A angle 90](B) \tkzGetPoint{D}
\tkzDefPointOnLine[pos=0.1](A,B)\tkzGetPoint{x1}
\tkzDefPointOnLine[pos=0.2](A,B)\tkzGetPoint{x2}
\tkzDefPointOnLine[pos=0.3](A,B)\tkzGetPoint{x3}
\tkzDefPointOnLine[pos=0.4](A,B)\tkzGetPoint{x4}
\tkzDefPointOnLine[pos=0.5](A,B)\tkzGetPoint{x5}
\tkzDefPointOnLine[pos=0.6](A,B)\tkzGetPoint{x6}
\tkzDefPointOnLine[pos=0.7](A,B)\tkzGetPoint{x7}
\tkzDefPointOnLine[pos=0.8](A,B)\tkzGetPoint{x8}
\tkzDefPointOnLine[pos=0.9](A,B)\tkzGetPoint{x9}
\tkzDefPointOnLine[pos=0.1](A,D)\tkzGetPoint{y1}
\tkzDefPointOnLine[pos=0.2](A,D)\tkzGetPoint{y2}
\tkzDefPointOnLine[pos=0.3](A,D)\tkzGetPoint{y3}
\tkzDefPointOnLine[pos=0.4](A,D)\tkzGetPoint{y4}
\tkzDefPointOnLine[pos=0.5](A,D)\tkzGetPoint{y5}
\tkzDefPointOnLine[pos=0.6](A,D)\tkzGetPoint{y6}
\tkzDefPointOnLine[pos=0.7](A,D)\tkzGetPoint{y7}
\tkzDefPointOnLine[pos=0.8](A,D)\tkzGetPoint{y8}
\tkzDefPointOnLine[pos=0.9](A,D)\tkzGetPoint{y9}
\tkzDefPointBy[translation= from A to B](y1)\tkzGetPoint{y11}
\tkzDefPointBy[translation= from A to B](y2)\tkzGetPoint{y12}
\tkzDefPointBy[translation= from A to B](y3)\tkzGetPoint{y13}
\tkzDefPointBy[translation= from A to B](y4)\tkzGetPoint{y14}
\tkzDefPointBy[translation= from A to B](y5)\tkzGetPoint{y15}
\tkzDefPointBy[translation= from A to B](y6)\tkzGetPoint{y16}
\tkzDefPointBy[translation= from A to B](y7)\tkzGetPoint{y17}
\tkzDefPointBy[translation= from A to B](y8)\tkzGetPoint{y18}
\tkzDefPointBy[translation= from A to B](y9)\tkzGetPoint{y19}
\tkzDefPointBy[translation= from A to B](D)\tkzGetPoint{C}
\tkzDefPointBy[translation= from A to D](x1)\tkzGetPoint{x11}
\tkzDefPointBy[translation= from A to D](x2)\tkzGetPoint{x12}
\tkzDefPointBy[translation= from A to D](x3)\tkzGetPoint{x13}
\tkzDefPointBy[translation= from A to D](x4)\tkzGetPoint{x14}
\tkzDefPointBy[translation= from A to D](x5)\tkzGetPoint{x15}
\tkzDefPointBy[translation= from A to D](x6)\tkzGetPoint{x16}
\tkzDefPointBy[translation= from A to D](x7)\tkzGetPoint{x17}
\tkzDefPointBy[translation= from A to D](x8)\tkzGetPoint{x18}
\tkzDefPointBy[translation= from A to D](x9)\tkzGetPoint{x19}

\tkzInterLL(x5,x15)(y5,y15)\tkzGetPoint{z}
\tkzInterLL(x1,x11)(y1,y11)\tkzGetPoint{z11}

\tkzDefPointBy[translation= from A to z11](y1)\tkzGetPoint{z12}
\tkzDefPointBy[translation= from A to z11](z12)\tkzGetPoint{z23}
\tkzDefPointBy[translation= from A to z11](z23)\tkzGetPoint{z34}
\tkzDefPointBy[translation= from A to z11](z34)\tkzGetPoint{z45}
\tkzDefPointBy[translation= from A to z11](z45)\tkzGetPoint{z56}
\tkzDefPointBy[translation= from A to z11](z56)\tkzGetPoint{z67}
\tkzDefPointBy[translation= from A to z11](z67)\tkzGetPoint{z78}
\tkzDefPointBy[translation= from A to z11](z78)\tkzGetPoint{z89}
\tkzDefPointBy[translation= from A to z11](x1)\tkzGetPoint{z21}
\tkzDefPointBy[translation= from A to z11](z21)\tkzGetPoint{z32}
\tkzDefPointBy[translation= from A to z11](z32)\tkzGetPoint{z43}
\tkzDefPointBy[translation= from A to z11](z43)\tkzGetPoint{z54}
\tkzDefPointBy[translation= from A to z11](z54)\tkzGetPoint{z65}
\tkzDefPointBy[translation= from A to z11](z65)\tkzGetPoint{z76}
\tkzDefPointBy[translation= from A to z11](z76)\tkzGetPoint{z87}
\tkzDefPointBy[translation= from A to z11](z87)\tkzGetPoint{z98}
\tkzDefPointBy[translation= from A to z11](z11)\tkzGetPoint{z22}
\tkzDefPointBy[translation= from A to z11](z22)\tkzGetPoint{z33}
\tkzDefPointBy[translation= from A to z11](z33)\tkzGetPoint{z44}
\tkzDefPointBy[translation= from A to z11](z44)\tkzGetPoint{z55}
\tkzDefPointBy[translation= from A to z11](z55)\tkzGetPoint{z66}
\tkzDefPointBy[translation= from A to z11](z66)\tkzGetPoint{z77}
\tkzDefPointBy[translation= from A to z11](z77)\tkzGetPoint{z88}
\tkzDefPointBy[translation= from A to z11](z88)\tkzGetPoint{z99}

\tkzInterLL(x2,x12)(y7,y17)\tkzGetPoint{z27}
\tkzInterLL(x7,x17)(y8,y18)\tkzGetPoint{z78}
\tkzInterLL(x2,x12)(y3,y13)\tkzGetPoint{z23}
\tkzInterLL(x7,x17)(y2,y12)\tkzGetPoint{z72}

\tkzInterLL(x1,x11)(y9,y19)\tkzGetPoint{z19}
\tkzInterLL(x2,x12)(y9,y19)\tkzGetPoint{z29}
\tkzInterLL(x3,x13)(y9,y19)\tkzGetPoint{z39}
\tkzInterLL(x4,x14)(y9,y19)\tkzGetPoint{z49}
\tkzInterLL(x5,x15)(y9,y19)\tkzGetPoint{z59}
\tkzInterLL(x6,x16)(y9,y19)\tkzGetPoint{z69}
\tkzInterLL(x7,x17)(y9,y19)\tkzGetPoint{z79}

\tkzInterLL(x9,x19)(y1,y11)\tkzGetPoint{z91}
\tkzInterLL(x9,x19)(y2,y12)\tkzGetPoint{z92}
\tkzInterLL(x9,x19)(y3,y13)\tkzGetPoint{z93}
\tkzInterLL(x9,x19)(y4,y14)\tkzGetPoint{z94}
\tkzInterLL(x9,x19)(y5,y15)\tkzGetPoint{z95}
\tkzInterLL(x9,x19)(y6,y16)\tkzGetPoint{z96}
\tkzInterLL(x9,x19)(y7,y17)\tkzGetPoint{z97}

\tkzDrawPolygon[fill=myred](x5,x9,z95,z55)
\tkzDrawPolygon[fill=mygreen](y9,z49,z44,y4)
\tkzDrawPolygon[fill=myblue](x4,A,y4,z44)
\tkzDrawPolygon[fill=myblue](z59,z99,z95,z55)

\tkzDrawLine[add=0 and 0.1,-Stealth](A, x9)
\tkzDrawLine[add=0 and 0.1,-Stealth](A, y9)

\tkzDrawSegment[line width=0.5pt](y1, z91)
\tkzDrawSegment[line width=0.5pt](y2, z92)
\tkzDrawSegment[line width=0.5pt](y3, z93)
\tkzDrawSegment(y4, z94)
\tkzDrawSegment(y5, z95)
\tkzDrawSegment[line width=0.5pt](y6, z96)
\tkzDrawSegment[line width=0.5pt](y7, z97)
\tkzDrawSegment[line width=0.5pt](y8, z98)
\tkzDrawSegment(y9, z99)

\tkzDrawSegment[line width=0.5pt](x1, z19)
\tkzDrawSegment[line width=0.5pt](x2, z29)
\tkzDrawSegment[line width=0.5pt](x3, z39)
\tkzDrawSegment(x4, z49)
\tkzDrawSegment(x5, z59)
\tkzDrawSegment[line width=0.5pt](x6, z69)
\tkzDrawSegment[line width=0.5pt](x7, z79)
\tkzDrawSegment[line width=0.5pt](x8, z89)
\tkzDrawSegment(x9, z99)

\tkzDrawLine[add=-0.2 and -0.2](x1,y1)
\tkzDrawLine[add=-0.2 and -0.2](A,z11)

\tkzDrawLine[add=-0.2 and -0.2](z12,z21)
\tkzDrawLine[add=-0.2 and -0.2](z11,z22)

\tkzDrawLine[add=-0.2 and -0.2](z23,z32)
\tkzDrawLine[add=-0.2 and -0.2](z22,z33)

\tkzDrawLine[add=-0.2 and -0.2](z34,z43)
\tkzDrawLine[add=-0.2 and -0.2](z33,z44)

\tkzDrawLine[add=-0.2 and -0.2](z45,z54)
\tkzDrawLine[add=-0.2 and -0.2](z44,z55)

\tkzDrawLine[add=-0.2 and -0.2](z56,z65)
\tkzDrawLine[add=-0.2 and -0.2](z55,z66)

\tkzDrawLine[add=-0.2 and -0.2](z67,z76)
\tkzDrawLine[add=-0.2 and -0.2](z66,z77)

\tkzDrawLine[add=-0.2 and -0.2](z78,z87)
\tkzDrawLine[add=-0.2 and -0.2](z77,z88)

\tkzDrawLine[add=-0.2 and -0.2](z89,z98)
\tkzDrawLine[add=-0.2 and -0.2](z88,z99)

\tkzLabelLine[pos=.97,right=-1.0mm](A,D){$j$}
\tkzLabelLine[pos=0.85,left](A,D){$2l\!+\!1$}
\tkzLabelLine[pos=0.55,left](A,D){$l\!+\!2$}
\tkzLabelLine[pos=0.45,left](A,D){$l\!+\!1$}
\tkzLabelLine[pos=0.35,left](A,D){$l$}
\tkzLabelLine[pos=0.05,left](A,D){$1$}

\tkzLabelLine[pos=.96,above](A,B){$i$}
\tkzLabelLine[below=-2mm](x5,x6){$l\!+\!2$}
\tkzLabelLine[below=-2mm](x4,x5){$l\!+\!1$}
\tkzLabelLine[below=0mm](x3,x4){$l$}
\tkzLabelLine[pos=0.05,below](A,B){$1$}
\tkzLabelLine[pos=0.95,below=-3mm](A,x9){$2l\!+\!1$}

\tkzLabelPoint[above left](z27){\Huge$\tilde{\textbf{A}}$}
\tkzLabelPoint[above left](z77){\Huge$\tilde{\textbf{D}}$}
\tkzLabelPoint[above left](z22){\Huge$\tilde{\textbf{C}}$}
\tkzLabelPoint[shift={(-5.5mm,19.7mm)}](z72){\Huge$\tilde{\textbf{B}}$}
\end{tikzpicture}
}
\caption{Odd case. ${\tilde{\text{A}}}\leqslant\tilde{\text{B}}$ because $\tilde{\text{C}}=\tilde{\text{D}}$ and $\tilde{\text{A}}+\tilde{\text{C}}\leqslant\tilde{\text{D}}+\tilde{\text{B}}$}
\label{fig:odd:1}
\end{figure}
--- that is, the number of arcs in the green area in Figure~\ref{fig:odd:1} (denoted by \(\tilde{\text{A}}\)) does not exceed the number of arcs in the red area (denoted by \(\tilde{\text{B}}\)).

\begin{figure}[h]
\centering
\resizebox{5.7cm}{!}{
\begin{tikzpicture}
[scale=1,vect/.style={->,shorten >=1pt,>=latex'}]
\tkzSetUpLine[line width=1pt]
\tkzSetUpCircle[color=black, line width=1pt]
\tkzSetUpPoint[size=4,circle,fill=black!]
\tkzSetUpArc[delta=0,color=black,line width=.75pt]

\tkzDefPoints{0/0/A, 7/0/B}
\tkzDefPointBy[rotation=center A angle 90](B) \tkzGetPoint{D}
\tkzDefPointOnLine[pos=0.1](A,B)\tkzGetPoint{x1}
\tkzDefPointOnLine[pos=0.2](A,B)\tkzGetPoint{x2}
\tkzDefPointOnLine[pos=0.3](A,B)\tkzGetPoint{x3}
\tkzDefPointOnLine[pos=0.4](A,B)\tkzGetPoint{x4}
\tkzDefPointOnLine[pos=0.5](A,B)\tkzGetPoint{x5}
\tkzDefPointOnLine[pos=0.6](A,B)\tkzGetPoint{x6}
\tkzDefPointOnLine[pos=0.7](A,B)\tkzGetPoint{x7}
\tkzDefPointOnLine[pos=0.8](A,B)\tkzGetPoint{x8}
\tkzDefPointOnLine[pos=0.9](A,B)\tkzGetPoint{x9}
\tkzDefPointOnLine[pos=0.1](A,D)\tkzGetPoint{y1}
\tkzDefPointOnLine[pos=0.2](A,D)\tkzGetPoint{y2}
\tkzDefPointOnLine[pos=0.3](A,D)\tkzGetPoint{y3}
\tkzDefPointOnLine[pos=0.4](A,D)\tkzGetPoint{y4}
\tkzDefPointOnLine[pos=0.5](A,D)\tkzGetPoint{y5}
\tkzDefPointOnLine[pos=0.6](A,D)\tkzGetPoint{y6}
\tkzDefPointOnLine[pos=0.7](A,D)\tkzGetPoint{y7}
\tkzDefPointOnLine[pos=0.8](A,D)\tkzGetPoint{y8}
\tkzDefPointOnLine[pos=0.9](A,D)\tkzGetPoint{y9}
\tkzDefPointBy[translation= from A to B](y1)\tkzGetPoint{y11}
\tkzDefPointBy[translation= from A to B](y2)\tkzGetPoint{y12}
\tkzDefPointBy[translation= from A to B](y3)\tkzGetPoint{y13}
\tkzDefPointBy[translation= from A to B](y4)\tkzGetPoint{y14}
\tkzDefPointBy[translation= from A to B](y5)\tkzGetPoint{y15}
\tkzDefPointBy[translation= from A to B](y6)\tkzGetPoint{y16}
\tkzDefPointBy[translation= from A to B](y7)\tkzGetPoint{y17}
\tkzDefPointBy[translation= from A to B](y8)\tkzGetPoint{y18}
\tkzDefPointBy[translation= from A to B](y9)\tkzGetPoint{y19}
\tkzDefPointBy[translation= from A to B](D)\tkzGetPoint{C}
\tkzDefPointBy[translation= from A to D](x1)\tkzGetPoint{x11}
\tkzDefPointBy[translation= from A to D](x2)\tkzGetPoint{x12}
\tkzDefPointBy[translation= from A to D](x3)\tkzGetPoint{x13}
\tkzDefPointBy[translation= from A to D](x4)\tkzGetPoint{x14}
\tkzDefPointBy[translation= from A to D](x5)\tkzGetPoint{x15}
\tkzDefPointBy[translation= from A to D](x6)\tkzGetPoint{x16}
\tkzDefPointBy[translation= from A to D](x7)\tkzGetPoint{x17}
\tkzDefPointBy[translation= from A to D](x8)\tkzGetPoint{x18}
\tkzDefPointBy[translation= from A to D](x9)\tkzGetPoint{x19}

\tkzInterLL(x5,x15)(y5,y15)\tkzGetPoint{z}
\tkzInterLL(x1,x11)(y1,y11)\tkzGetPoint{z11}

\tkzDefPointBy[translation= from A to z11](y1)\tkzGetPoint{z12}
\tkzDefPointBy[translation= from A to z11](z12)\tkzGetPoint{z23}
\tkzDefPointBy[translation= from A to z11](z23)\tkzGetPoint{z34}
\tkzDefPointBy[translation= from A to z11](z34)\tkzGetPoint{z45}
\tkzDefPointBy[translation= from A to z11](z45)\tkzGetPoint{z56}
\tkzDefPointBy[translation= from A to z11](z56)\tkzGetPoint{z67}
\tkzDefPointBy[translation= from A to z11](z67)\tkzGetPoint{z78}
\tkzDefPointBy[translation= from A to z11](z78)\tkzGetPoint{z89}
\tkzDefPointBy[translation= from A to z11](x1)\tkzGetPoint{z21}
\tkzDefPointBy[translation= from A to z11](z21)\tkzGetPoint{z32}
\tkzDefPointBy[translation= from A to z11](z32)\tkzGetPoint{z43}
\tkzDefPointBy[translation= from A to z11](z43)\tkzGetPoint{z54}
\tkzDefPointBy[translation= from A to z11](z54)\tkzGetPoint{z65}
\tkzDefPointBy[translation= from A to z11](z65)\tkzGetPoint{z76}
\tkzDefPointBy[translation= from A to z11](z76)\tkzGetPoint{z87}
\tkzDefPointBy[translation= from A to z11](z87)\tkzGetPoint{z98}
\tkzDefPointBy[translation= from A to z11](z11)\tkzGetPoint{z22}
\tkzDefPointBy[translation= from A to z11](z22)\tkzGetPoint{z33}
\tkzDefPointBy[translation= from A to z11](z33)\tkzGetPoint{z44}
\tkzDefPointBy[translation= from A to z11](z44)\tkzGetPoint{z55}
\tkzDefPointBy[translation= from A to z11](z55)\tkzGetPoint{z66}
\tkzDefPointBy[translation= from A to z11](z66)\tkzGetPoint{z77}
\tkzDefPointBy[translation= from A to z11](z77)\tkzGetPoint{z88}
\tkzDefPointBy[translation= from A to z11](z88)\tkzGetPoint{z99}

\tkzInterLL(x2,x12)(y8,y18)\tkzGetPoint{z28}
\tkzInterLL(x7,x17)(y8,y18)\tkzGetPoint{z78}
\tkzInterLL(x2,x12)(y3,y13)\tkzGetPoint{z23}
\tkzInterLL(x7,x17)(y3,y13)\tkzGetPoint{z73}

\tkzInterLL(x1,x11)(y9,y19)\tkzGetPoint{z19}
\tkzInterLL(x2,x12)(y9,y19)\tkzGetPoint{z29}
\tkzInterLL(x3,x13)(y9,y19)\tkzGetPoint{z39}
\tkzInterLL(x4,x14)(y9,y19)\tkzGetPoint{z49}
\tkzInterLL(x5,x15)(y9,y19)\tkzGetPoint{z59}
\tkzInterLL(x6,x16)(y9,y19)\tkzGetPoint{z69}
\tkzInterLL(x7,x17)(y9,y19)\tkzGetPoint{z79}

\tkzInterLL(x9,x19)(y1,y11)\tkzGetPoint{z91}
\tkzInterLL(x9,x19)(y2,y12)\tkzGetPoint{z92}
\tkzInterLL(x9,x19)(y3,y13)\tkzGetPoint{z93}
\tkzInterLL(x9,x19)(y4,y14)\tkzGetPoint{z94}
\tkzInterLL(x9,x19)(y5,y15)\tkzGetPoint{z95}
\tkzInterLL(x9,x19)(y6,y16)\tkzGetPoint{z96}
\tkzInterLL(x9,x19)(y7,y17)\tkzGetPoint{z97}


\tkzDrawPolygon[fill=mysand](z59,z55,z45,z44,y4,y9)

\tkzDrawLine[add=0 and 0.1,-Stealth](A, x9)
\tkzDrawLine[add=0 and 0.1,-Stealth](A, y9)

\tkzDrawSegment[line width=0.5pt](y1, z91)
\tkzDrawSegment[line width=0.5pt](y2, z92)
\tkzDrawSegment[line width=0.5pt](y3, z93)
\tkzDrawSegment(y4, z94)
\tkzDrawSegment(y5, z95)
\tkzDrawSegment[line width=0.5pt](y6, z96)
\tkzDrawSegment[line width=0.5pt](y7, z97)
\tkzDrawSegment[line width=0.5pt](y8, z98)
\tkzDrawSegment(y9, z99)

\tkzDrawSegment[line width=0.5pt](x1, z19)
\tkzDrawSegment[line width=0.5pt](x2, z29)
\tkzDrawSegment[line width=0.5pt](x3, z39)
\tkzDrawSegment(x4, z49)
\tkzDrawSegment(x5, z59)
\tkzDrawSegment[line width=0.5pt](x6, z69)
\tkzDrawSegment[line width=0.5pt](x7, z79)
\tkzDrawSegment[line width=0.5pt](x8, z89)
\tkzDrawSegment(x9, z99)

\tkzDrawLine[add=-0.2 and -0.2](x1,y1)
\tkzDrawLine[add=-0.2 and -0.2](A,z11)

\tkzDrawLine[add=-0.2 and -0.2](z12,z21)
\tkzDrawLine[add=-0.2 and -0.2](z11,z22)

\tkzDrawLine[add=-0.2 and -0.2](z23,z32)
\tkzDrawLine[add=-0.2 and -0.2](z22,z33)

\tkzDrawLine[add=-0.2 and -0.2](z34,z43)
\tkzDrawLine[add=-0.2 and -0.2](z33,z44)

\tkzDrawLine[add=-0.2 and -0.2](z45,z54)
\tkzDrawLine[add=-0.2 and -0.2](z44,z55)

\tkzDrawLine[add=-0.2 and -0.2](z56,z65)
\tkzDrawLine[add=-0.2 and -0.2](z55,z66)

\tkzDrawLine[add=-0.2 and -0.2](z67,z76)
\tkzDrawLine[add=-0.2 and -0.2](z66,z77)

\tkzDrawLine[add=-0.2 and -0.2](z78,z87)
\tkzDrawLine[add=-0.2 and -0.2](z77,z88)

\tkzDrawLine[add=-0.2 and -0.2](z89,z98)
\tkzDrawLine[add=-0.2 and -0.2](z88,z99)

\tkzLabelLine[pos=.97,right=-1.0mm](A,D){$j$}
\tkzLabelLine[pos=0.85,left](A,D){$2l\!+\!1$}
\tkzLabelLine[pos=0.55,left](A,D){$l\!+\!2$}
\tkzLabelLine[pos=0.45,left](A,D){$l\!+\!1$}
\tkzLabelLine[pos=0.35,left](A,D){$l$}
\tkzLabelLine[pos=0.05,left](A,D){$1$}

\tkzLabelLine[pos=.96,above](A,B){$i$}
\tkzLabelLine[below=-2mm](x5,x6){$l\!+\!2$}
\tkzLabelLine[below=-2mm](x4,x5){$l\!+\!1$}
\tkzLabelLine[below=0mm](x3,x4){$l$}
\tkzLabelLine[pos=0.05,below](A,B){$1$}
\tkzLabelLine[pos=0.95,below=-3mm](A,x9){$2l\!+\!1$}

\tkzLabelPoint[above left](z27){\Huge$\tilde{\textbf{G}}$}
\end{tikzpicture}
} 
\resizebox{5.7cm}{!}{
\begin{tikzpicture}
[scale=1,vect/.style={->,shorten >=1pt,>=latex'}]
\tkzSetUpLine[line width=1pt]
\tkzSetUpCircle[color=black, line width=1pt]
\tkzSetUpPoint[size=4,circle,fill=black!]
\tkzSetUpArc[delta=0,color=black,line width=.75pt]

\tkzDefPoints{0/0/A, 7/0/B}
\tkzDefPointBy[rotation=center A angle 90](B) \tkzGetPoint{D}
\tkzDefPointOnLine[pos=0.1](A,B)\tkzGetPoint{x1}
\tkzDefPointOnLine[pos=0.2](A,B)\tkzGetPoint{x2}
\tkzDefPointOnLine[pos=0.3](A,B)\tkzGetPoint{x3}
\tkzDefPointOnLine[pos=0.4](A,B)\tkzGetPoint{x4}
\tkzDefPointOnLine[pos=0.5](A,B)\tkzGetPoint{x5}
\tkzDefPointOnLine[pos=0.6](A,B)\tkzGetPoint{x6}
\tkzDefPointOnLine[pos=0.7](A,B)\tkzGetPoint{x7}
\tkzDefPointOnLine[pos=0.8](A,B)\tkzGetPoint{x8}
\tkzDefPointOnLine[pos=0.9](A,B)\tkzGetPoint{x9}
\tkzDefPointOnLine[pos=0.1](A,D)\tkzGetPoint{y1}
\tkzDefPointOnLine[pos=0.2](A,D)\tkzGetPoint{y2}
\tkzDefPointOnLine[pos=0.3](A,D)\tkzGetPoint{y3}
\tkzDefPointOnLine[pos=0.4](A,D)\tkzGetPoint{y4}
\tkzDefPointOnLine[pos=0.5](A,D)\tkzGetPoint{y5}
\tkzDefPointOnLine[pos=0.6](A,D)\tkzGetPoint{y6}
\tkzDefPointOnLine[pos=0.7](A,D)\tkzGetPoint{y7}
\tkzDefPointOnLine[pos=0.8](A,D)\tkzGetPoint{y8}
\tkzDefPointOnLine[pos=0.9](A,D)\tkzGetPoint{y9}
\tkzDefPointBy[translation= from A to B](y1)\tkzGetPoint{y11}
\tkzDefPointBy[translation= from A to B](y2)\tkzGetPoint{y12}
\tkzDefPointBy[translation= from A to B](y3)\tkzGetPoint{y13}
\tkzDefPointBy[translation= from A to B](y4)\tkzGetPoint{y14}
\tkzDefPointBy[translation= from A to B](y5)\tkzGetPoint{y15}
\tkzDefPointBy[translation= from A to B](y6)\tkzGetPoint{y16}
\tkzDefPointBy[translation= from A to B](y7)\tkzGetPoint{y17}
\tkzDefPointBy[translation= from A to B](y8)\tkzGetPoint{y18}
\tkzDefPointBy[translation= from A to B](y9)\tkzGetPoint{y19}
\tkzDefPointBy[translation= from A to B](D)\tkzGetPoint{C}
\tkzDefPointBy[translation= from A to D](x1)\tkzGetPoint{x11}
\tkzDefPointBy[translation= from A to D](x2)\tkzGetPoint{x12}
\tkzDefPointBy[translation= from A to D](x3)\tkzGetPoint{x13}
\tkzDefPointBy[translation= from A to D](x4)\tkzGetPoint{x14}
\tkzDefPointBy[translation= from A to D](x5)\tkzGetPoint{x15}
\tkzDefPointBy[translation= from A to D](x6)\tkzGetPoint{x16}
\tkzDefPointBy[translation= from A to D](x7)\tkzGetPoint{x17}
\tkzDefPointBy[translation= from A to D](x8)\tkzGetPoint{x18}
\tkzDefPointBy[translation= from A to D](x9)\tkzGetPoint{x19}

\tkzInterLL(x5,x15)(y5,y15)\tkzGetPoint{z}
\tkzInterLL(x1,x11)(y1,y11)\tkzGetPoint{z11}

\tkzDefPointBy[translation= from A to z11](y1)\tkzGetPoint{z12}
\tkzDefPointBy[translation= from A to z11](z12)\tkzGetPoint{z23}
\tkzDefPointBy[translation= from A to z11](z23)\tkzGetPoint{z34}
\tkzDefPointBy[translation= from A to z11](z34)\tkzGetPoint{z45}
\tkzDefPointBy[translation= from A to z11](z45)\tkzGetPoint{z56}
\tkzDefPointBy[translation= from A to z11](z56)\tkzGetPoint{z67}
\tkzDefPointBy[translation= from A to z11](z67)\tkzGetPoint{z78}
\tkzDefPointBy[translation= from A to z11](z78)\tkzGetPoint{z89}
\tkzDefPointBy[translation= from A to z11](x1)\tkzGetPoint{z21}
\tkzDefPointBy[translation= from A to z11](z21)\tkzGetPoint{z32}
\tkzDefPointBy[translation= from A to z11](z32)\tkzGetPoint{z43}
\tkzDefPointBy[translation= from A to z11](z43)\tkzGetPoint{z54}
\tkzDefPointBy[translation= from A to z11](z54)\tkzGetPoint{z65}
\tkzDefPointBy[translation= from A to z11](z65)\tkzGetPoint{z76}
\tkzDefPointBy[translation= from A to z11](z76)\tkzGetPoint{z87}
\tkzDefPointBy[translation= from A to z11](z87)\tkzGetPoint{z98}
\tkzDefPointBy[translation= from A to z11](z11)\tkzGetPoint{z22}
\tkzDefPointBy[translation= from A to z11](z22)\tkzGetPoint{z33}
\tkzDefPointBy[translation= from A to z11](z33)\tkzGetPoint{z44}
\tkzDefPointBy[translation= from A to z11](z44)\tkzGetPoint{z55}
\tkzDefPointBy[translation= from A to z11](z55)\tkzGetPoint{z66}
\tkzDefPointBy[translation= from A to z11](z66)\tkzGetPoint{z77}
\tkzDefPointBy[translation= from A to z11](z77)\tkzGetPoint{z88}
\tkzDefPointBy[translation= from A to z11](z88)\tkzGetPoint{z99}

\tkzInterLL(x2,x12)(y8,y18)\tkzGetPoint{z28}
\tkzInterLL(x7,x17)(y8,y18)\tkzGetPoint{z78}
\tkzInterLL(x2,x12)(y3,y13)\tkzGetPoint{z23}
\tkzInterLL(x7,x17)(y3,y13)\tkzGetPoint{z73}

\tkzInterLL(x1,x11)(y9,y19)\tkzGetPoint{z19}
\tkzInterLL(x2,x12)(y9,y19)\tkzGetPoint{z29}
\tkzInterLL(x3,x13)(y9,y19)\tkzGetPoint{z39}
\tkzInterLL(x4,x14)(y9,y19)\tkzGetPoint{z49}
\tkzInterLL(x5,x15)(y9,y19)\tkzGetPoint{z59}
\tkzInterLL(x6,x16)(y9,y19)\tkzGetPoint{z69}
\tkzInterLL(x7,x17)(y9,y19)\tkzGetPoint{z79}

\tkzInterLL(x9,x19)(y1,y11)\tkzGetPoint{z91}
\tkzInterLL(x9,x19)(y2,y12)\tkzGetPoint{z92}
\tkzInterLL(x9,x19)(y3,y13)\tkzGetPoint{z93}
\tkzInterLL(x9,x19)(y4,y14)\tkzGetPoint{z94}
\tkzInterLL(x9,x19)(y5,y15)\tkzGetPoint{z95}
\tkzInterLL(x9,x19)(y6,y16)\tkzGetPoint{z96}
\tkzInterLL(x9,x19)(y7,y17)\tkzGetPoint{z97}

\tkzDrawPolygon[fill=myred](x5,x9,z95,z55)
\tkzDrawPolygon[fill=myred](y9,z59,z55,y5)
\tkzDrawPolygon[fill=myred!50](y4,z44,z45,y5)

\tkzDrawLine[add=0 and 0.1,-Stealth](A, x9)
\tkzDrawLine[add=0 and 0.1,-Stealth](A, y9)

\tkzDrawSegment[line width=0.5pt](y1, z91)
\tkzDrawSegment[line width=0.5pt](y2, z92)
\tkzDrawSegment[line width=0.5pt](y3, z93)
\tkzDrawSegment(y4, z94)
\tkzDrawSegment(y5, z95)
\tkzDrawSegment[line width=0.5pt](y6, z96)
\tkzDrawSegment[line width=0.5pt](y7, z97)
\tkzDrawSegment[line width=0.5pt](y8, z98)
\tkzDrawSegment(y9, z99)

\tkzDrawSegment[line width=0.5pt](x1, z19)
\tkzDrawSegment[line width=0.5pt](x2, z29)
\tkzDrawSegment[line width=0.5pt](x3, z39)
\tkzDrawSegment(x4, z49)
\tkzDrawSegment(x5, z59)
\tkzDrawSegment[line width=0.5pt](x6, z69)
\tkzDrawSegment[line width=0.5pt](x7, z79)
\tkzDrawSegment[line width=0.5pt](x8, z89)
\tkzDrawSegment(x9, z99)

\tkzDrawLine[add=-0.2 and -0.2](x1,y1)
\tkzDrawLine[add=-0.2 and -0.2](A,z11)

\tkzDrawLine[add=-0.2 and -0.2](z12,z21)
\tkzDrawLine[add=-0.2 and -0.2](z11,z22)

\tkzDrawLine[add=-0.2 and -0.2](z23,z32)
\tkzDrawLine[add=-0.2 and -0.2](z22,z33)

\tkzDrawLine[add=-0.2 and -0.2](z34,z43)
\tkzDrawLine[add=-0.2 and -0.2](z33,z44)

\tkzDrawLine[add=-0.2 and -0.2](z45,z54)
\tkzDrawLine[add=-0.2 and -0.2](z44,z55)

\tkzDrawLine[add=-0.2 and -0.2](z56,z65)
\tkzDrawLine[add=-0.2 and -0.2](z55,z66)

\tkzDrawLine[add=-0.2 and -0.2](z67,z76)
\tkzDrawLine[add=-0.2 and -0.2](z66,z77)

\tkzDrawLine[add=-0.2 and -0.2](z78,z87)
\tkzDrawLine[add=-0.2 and -0.2](z77,z88)

\tkzDrawLine[add=-0.2 and -0.2](z89,z98)
\tkzDrawLine[add=-0.2 and -0.2](z88,z99)

\tkzLabelLine[pos=.97,right=-1.0mm](A,D){$j$}
\tkzLabelLine[pos=0.85,left](A,D){$2l\!+\!1$}
\tkzLabelLine[pos=0.55,left](A,D){$l\!+\!2$}
\tkzLabelLine[pos=0.45,left](A,D){$l\!+\!1$}
\tkzLabelLine[pos=0.35,left](A,D){$l$}
\tkzLabelLine[pos=0.05,left](A,D){$1$}

\tkzLabelLine[pos=.96,above](A,B){$i$}
\tkzLabelLine[below=-2mm](x5,x6){$l\!+\!2$}
\tkzLabelLine[below=-2mm](x4,x5){$l\!+\!1$}
\tkzLabelLine[below=0mm](x3,x4){$l$}
\tkzLabelLine[pos=0.05,below](A,B){$1$}
\tkzLabelLine[pos=0.95,below=-3mm](A,x9){$2l\!+\!1$}

\tkzLabelPoint[above left](z27){\Huge$\tilde{\textbf{H}}$}
\tkzLabelPoint[shift={(-11mm,8.4mm)}](z34){\Large$\tilde{\textbf{I}}$}
\tkzLabelPoint[shift={(-5.5mm,19.7mm)}](z72){\Huge$\tilde{\textbf{B}}$}
\end{tikzpicture}
} 
\resizebox{5.7cm}{!}{
\begin{tikzpicture}
[scale=1,vect/.style={->,shorten >=1pt,>=latex'}]
\tkzSetUpLine[line width=1pt]
\tkzSetUpCircle[color=black, line width=1pt]
\tkzSetUpPoint[size=4,circle,fill=black!]
\tkzSetUpArc[delta=0,color=black,line width=.75pt]

\tkzDefPoints{0/0/A, 7/0/B}
\tkzDefPointBy[rotation=center A angle 90](B) \tkzGetPoint{D}
\tkzDefPointOnLine[pos=0.1](A,B)\tkzGetPoint{x1}
\tkzDefPointOnLine[pos=0.2](A,B)\tkzGetPoint{x2}
\tkzDefPointOnLine[pos=0.3](A,B)\tkzGetPoint{x3}
\tkzDefPointOnLine[pos=0.4](A,B)\tkzGetPoint{x4}
\tkzDefPointOnLine[pos=0.5](A,B)\tkzGetPoint{x5}
\tkzDefPointOnLine[pos=0.6](A,B)\tkzGetPoint{x6}
\tkzDefPointOnLine[pos=0.7](A,B)\tkzGetPoint{x7}
\tkzDefPointOnLine[pos=0.8](A,B)\tkzGetPoint{x8}
\tkzDefPointOnLine[pos=0.9](A,B)\tkzGetPoint{x9}
\tkzDefPointOnLine[pos=0.1](A,D)\tkzGetPoint{y1}
\tkzDefPointOnLine[pos=0.2](A,D)\tkzGetPoint{y2}
\tkzDefPointOnLine[pos=0.3](A,D)\tkzGetPoint{y3}
\tkzDefPointOnLine[pos=0.4](A,D)\tkzGetPoint{y4}
\tkzDefPointOnLine[pos=0.5](A,D)\tkzGetPoint{y5}
\tkzDefPointOnLine[pos=0.6](A,D)\tkzGetPoint{y6}
\tkzDefPointOnLine[pos=0.7](A,D)\tkzGetPoint{y7}
\tkzDefPointOnLine[pos=0.8](A,D)\tkzGetPoint{y8}
\tkzDefPointOnLine[pos=0.9](A,D)\tkzGetPoint{y9}
\tkzDefPointBy[translation= from A to B](y1)\tkzGetPoint{y11}
\tkzDefPointBy[translation= from A to B](y2)\tkzGetPoint{y12}
\tkzDefPointBy[translation= from A to B](y3)\tkzGetPoint{y13}
\tkzDefPointBy[translation= from A to B](y4)\tkzGetPoint{y14}
\tkzDefPointBy[translation= from A to B](y5)\tkzGetPoint{y15}
\tkzDefPointBy[translation= from A to B](y6)\tkzGetPoint{y16}
\tkzDefPointBy[translation= from A to B](y7)\tkzGetPoint{y17}
\tkzDefPointBy[translation= from A to B](y8)\tkzGetPoint{y18}
\tkzDefPointBy[translation= from A to B](y9)\tkzGetPoint{y19}
\tkzDefPointBy[translation= from A to B](D)\tkzGetPoint{C}
\tkzDefPointBy[translation= from A to D](x1)\tkzGetPoint{x11}
\tkzDefPointBy[translation= from A to D](x2)\tkzGetPoint{x12}
\tkzDefPointBy[translation= from A to D](x3)\tkzGetPoint{x13}
\tkzDefPointBy[translation= from A to D](x4)\tkzGetPoint{x14}
\tkzDefPointBy[translation= from A to D](x5)\tkzGetPoint{x15}
\tkzDefPointBy[translation= from A to D](x6)\tkzGetPoint{x16}
\tkzDefPointBy[translation= from A to D](x7)\tkzGetPoint{x17}
\tkzDefPointBy[translation= from A to D](x8)\tkzGetPoint{x18}
\tkzDefPointBy[translation= from A to D](x9)\tkzGetPoint{x19}

\tkzInterLL(x5,x15)(y5,y15)\tkzGetPoint{z}
\tkzInterLL(x1,x11)(y1,y11)\tkzGetPoint{z11}

\tkzDefPointBy[translation= from A to z11](y1)\tkzGetPoint{z12}
\tkzDefPointBy[translation= from A to z11](z12)\tkzGetPoint{z23}
\tkzDefPointBy[translation= from A to z11](z23)\tkzGetPoint{z34}
\tkzDefPointBy[translation= from A to z11](z34)\tkzGetPoint{z45}
\tkzDefPointBy[translation= from A to z11](z45)\tkzGetPoint{z56}
\tkzDefPointBy[translation= from A to z11](z56)\tkzGetPoint{z67}
\tkzDefPointBy[translation= from A to z11](z67)\tkzGetPoint{z78}
\tkzDefPointBy[translation= from A to z11](z78)\tkzGetPoint{z89}
\tkzDefPointBy[translation= from A to z11](x1)\tkzGetPoint{z21}
\tkzDefPointBy[translation= from A to z11](z21)\tkzGetPoint{z32}
\tkzDefPointBy[translation= from A to z11](z32)\tkzGetPoint{z43}
\tkzDefPointBy[translation= from A to z11](z43)\tkzGetPoint{z54}
\tkzDefPointBy[translation= from A to z11](z54)\tkzGetPoint{z65}
\tkzDefPointBy[translation= from A to z11](z65)\tkzGetPoint{z76}
\tkzDefPointBy[translation= from A to z11](z76)\tkzGetPoint{z87}
\tkzDefPointBy[translation= from A to z11](z87)\tkzGetPoint{z98}
\tkzDefPointBy[translation= from A to z11](z11)\tkzGetPoint{z22}
\tkzDefPointBy[translation= from A to z11](z22)\tkzGetPoint{z33}
\tkzDefPointBy[translation= from A to z11](z33)\tkzGetPoint{z44}
\tkzDefPointBy[translation= from A to z11](z44)\tkzGetPoint{z55}
\tkzDefPointBy[translation= from A to z11](z55)\tkzGetPoint{z66}
\tkzDefPointBy[translation= from A to z11](z66)\tkzGetPoint{z77}
\tkzDefPointBy[translation= from A to z11](z77)\tkzGetPoint{z88}
\tkzDefPointBy[translation= from A to z11](z88)\tkzGetPoint{z99}

\tkzInterLL(x2,x12)(y8,y18)\tkzGetPoint{z28}
\tkzInterLL(x7,x17)(y8,y18)\tkzGetPoint{z78}
\tkzInterLL(x2,x12)(y3,y13)\tkzGetPoint{z23}
\tkzInterLL(x7,x17)(y3,y13)\tkzGetPoint{z73}

\tkzInterLL(x1,x11)(y9,y19)\tkzGetPoint{z19}
\tkzInterLL(x2,x12)(y9,y19)\tkzGetPoint{z29}
\tkzInterLL(x3,x13)(y9,y19)\tkzGetPoint{z39}
\tkzInterLL(x4,x14)(y9,y19)\tkzGetPoint{z49}
\tkzInterLL(x5,x15)(y9,y19)\tkzGetPoint{z59}
\tkzInterLL(x6,x16)(y9,y19)\tkzGetPoint{z69}
\tkzInterLL(x7,x17)(y9,y19)\tkzGetPoint{z79}

\tkzInterLL(x9,x19)(y1,y11)\tkzGetPoint{z91}
\tkzInterLL(x9,x19)(y2,y12)\tkzGetPoint{z92}
\tkzInterLL(x9,x19)(y3,y13)\tkzGetPoint{z93}
\tkzInterLL(x9,x19)(y4,y14)\tkzGetPoint{z94}
\tkzInterLL(x9,x19)(y5,y15)\tkzGetPoint{z95}
\tkzInterLL(x9,x19)(y6,y16)\tkzGetPoint{z96}
\tkzInterLL(x9,x19)(y7,y17)\tkzGetPoint{z97}

\tkzDrawPolygon[fill=mygreen](y9,z49,z44,y4)
\tkzDrawPolygon[fill=mygreen!50](z45,z55,z59,z49)

\tkzDrawLine[add=0 and 0.1,-Stealth](A, x9)
\tkzDrawLine[add=0 and 0.1,-Stealth](A, y9)

\tkzDrawSegment[line width=0.5pt](y1, z91)
\tkzDrawSegment[line width=0.5pt](y2, z92)
\tkzDrawSegment[line width=0.5pt](y3, z93)
\tkzDrawSegment(y4, z94)
\tkzDrawSegment(y5, z95)
\tkzDrawSegment[line width=0.5pt](y6, z96)
\tkzDrawSegment[line width=0.5pt](y7, z97)
\tkzDrawSegment[line width=0.5pt](y8, z98)
\tkzDrawSegment(y9, z99)

\tkzDrawSegment[line width=0.5pt](x1, z19)
\tkzDrawSegment[line width=0.5pt](x2, z29)
\tkzDrawSegment[line width=0.5pt](x3, z39)
\tkzDrawSegment(x4, z49)
\tkzDrawSegment(x5, z59)
\tkzDrawSegment[line width=0.5pt](x6, z69)
\tkzDrawSegment[line width=0.5pt](x7, z79)
\tkzDrawSegment[line width=0.5pt](x8, z89)
\tkzDrawSegment(x9, z99)

\tkzDrawLine[add=-0.2 and -0.2](x1,y1)
\tkzDrawLine[add=-0.2 and -0.2](A,z11)

\tkzDrawLine[add=-0.2 and -0.2](z12,z21)
\tkzDrawLine[add=-0.2 and -0.2](z11,z22)

\tkzDrawLine[add=-0.2 and -0.2](z23,z32)
\tkzDrawLine[add=-0.2 and -0.2](z22,z33)

\tkzDrawLine[add=-0.2 and -0.2](z34,z43)
\tkzDrawLine[add=-0.2 and -0.2](z33,z44)

\tkzDrawLine[add=-0.2 and -0.2](z45,z54)
\tkzDrawLine[add=-0.2 and -0.2](z44,z55)

\tkzDrawLine[add=-0.2 and -0.2](z56,z65)
\tkzDrawLine[add=-0.2 and -0.2](z55,z66)

\tkzDrawLine[add=-0.2 and -0.2](z67,z76)
\tkzDrawLine[add=-0.2 and -0.2](z66,z77)

\tkzDrawLine[add=-0.2 and -0.2](z78,z87)
\tkzDrawLine[add=-0.2 and -0.2](z77,z88)

\tkzDrawLine[add=-0.2 and -0.2](z89,z98)
\tkzDrawLine[add=-0.2 and -0.2](z88,z99)

\tkzLabelLine[pos=.97,right=-1.0mm](A,D){$j$}
\tkzLabelLine[pos=0.85,left](A,D){$2l\!+\!1$}
\tkzLabelLine[pos=0.55,left](A,D){$l\!+\!2$}
\tkzLabelLine[pos=0.45,left](A,D){$l\!+\!1$}
\tkzLabelLine[pos=0.35,left](A,D){$l$}
\tkzLabelLine[pos=0.05,left](A,D){$1$}

\tkzLabelLine[pos=.96,above](A,B){$i$}
\tkzLabelLine[below=-2mm](x5,x6){$l\!+\!2$}
\tkzLabelLine[below=-2mm](x4,x5){$l\!+\!1$}
\tkzLabelLine[below=0mm](x3,x4){$l$}
\tkzLabelLine[pos=0.05,below](A,B){$1$}
\tkzLabelLine[pos=0.95,below=-3mm](A,x9){$2l\!+\!1$}

\tkzLabelPoint[above left](z27){\Huge$\tilde{\textbf{A}}$}
\tkzLabelPoint[shift={(-4.3mm,16mm)}](z56){\Large$\tilde{\textbf{J}}$}
\end{tikzpicture}
}
\caption{Odd case. $2\tilde{\text{G}}=\tilde{\text{I}}+\tilde{\text{H}}+\tilde{\text{A}}+\tilde{\text{J}}
\leqslant{l}+(\tilde{\text{H}}+\tilde{\text{B}})+{l}={l} + ({l}\,{+}\,{1}){l} + {l}$}
\label{fig:odd:2}
\end{figure}

Now, using (\ref{Lemma.3.odd'}), we can write (see Figure~\ref{fig:odd:2})
\begin{align*}
2\,&|\{\la{i},{j}\ra\in{E}^{<}:\:{i}\leq{l}\!+\!1,\:{j}>{l}\}|\ =\  
2\,|\{\la{i},{j}\ra\in{E}:{i}\leq{l}\!+\!1,\:{j}>{l}\}|\ =\\
=\ & |\{\la{i},{j}\ra\in{E}:{i}\leq{l}\!+\!1,\:{j}={l}\!+\!1\}|+
|\{\la{i},{j}\ra\in{E}:{i}\leq{l}\!+\!1,\:{j}>{l}\!+\!1\}|\ +\\
+\ &|\{\la{i},{j}\ra\in{E}:{i}\leq{l},\:{j}>{l}\}| + 
|\{\la{i},{j}\ra\in{E}:{i}={l}\!+\!1,\:{j}>{l}\}|\ \leq \\
\leq\ &{l} + |\{\la{i},{j}\ra\in{E}:{i}\leq{l}\!+\!1,\:{j}>{l}\!+\!1\}| + 
|\{\la{i},{j}\ra\in{E}:{j}\leq{l}\!+\!1,\:{i}>{l}\!+\!1\}| + {l}\ =\\
=\ &{l} + |\{\la{i},{j}\ra\in{E}:{i}\leq{l}\!+\!1,\:{j}>{l}\!+\!1\}| + 
|\{\la{j},{i}\ra\in{E}:{i}\leq{l}\!+\!1,\:{j}>{l}\!+\!1\}| + {l}\ =\\
=\ &{l} + ({l}+1){l} + {l}\ =\ {l}({l}+3),
\end{align*}
that is, 
\begin{equation}\label{Lemma.4.odd''}\textstyle
    |\{\la{i},{j}\ra\in{E}^{<},\:{i}\leq{l}\!+\!1,\:{j}>{l}\}|\ \leq\ \frac{{l}({l}+3)}{2}.
\end{equation}

\begin{figure}[h]
\centering
\resizebox{6cm}{!}{
\begin{tikzpicture}
[scale=1,vect/.style={->,shorten >=1pt,>=latex'}]
\tkzSetUpLine[line width=1pt]
\tkzSetUpCircle[color=black, line width=1pt]
\tkzSetUpPoint[size=4,circle,fill=black!]
\tkzSetUpArc[delta=0,color=black,line width=.75pt]

\tkzDefPoints{0/0/A, 7/0/B}
\tkzDefPointBy[rotation=center A angle 90](B) \tkzGetPoint{D}
\tkzDefPointOnLine[pos=0.1](A,B)\tkzGetPoint{x1}
\tkzDefPointOnLine[pos=0.2](A,B)\tkzGetPoint{x2}
\tkzDefPointOnLine[pos=0.3](A,B)\tkzGetPoint{x3}
\tkzDefPointOnLine[pos=0.4](A,B)\tkzGetPoint{x4}
\tkzDefPointOnLine[pos=0.5](A,B)\tkzGetPoint{x5}
\tkzDefPointOnLine[pos=0.6](A,B)\tkzGetPoint{x6}
\tkzDefPointOnLine[pos=0.7](A,B)\tkzGetPoint{x7}
\tkzDefPointOnLine[pos=0.8](A,B)\tkzGetPoint{x8}
\tkzDefPointOnLine[pos=0.9](A,B)\tkzGetPoint{x9}
\tkzDefPointOnLine[pos=0.1](A,D)\tkzGetPoint{y1}
\tkzDefPointOnLine[pos=0.2](A,D)\tkzGetPoint{y2}
\tkzDefPointOnLine[pos=0.3](A,D)\tkzGetPoint{y3}
\tkzDefPointOnLine[pos=0.4](A,D)\tkzGetPoint{y4}
\tkzDefPointOnLine[pos=0.5](A,D)\tkzGetPoint{y5}
\tkzDefPointOnLine[pos=0.6](A,D)\tkzGetPoint{y6}
\tkzDefPointOnLine[pos=0.7](A,D)\tkzGetPoint{y7}
\tkzDefPointOnLine[pos=0.8](A,D)\tkzGetPoint{y8}
\tkzDefPointOnLine[pos=0.9](A,D)\tkzGetPoint{y9}
\tkzDefPointBy[translation= from A to B](y1)\tkzGetPoint{y11}
\tkzDefPointBy[translation= from A to B](y2)\tkzGetPoint{y12}
\tkzDefPointBy[translation= from A to B](y3)\tkzGetPoint{y13}
\tkzDefPointBy[translation= from A to B](y4)\tkzGetPoint{y14}
\tkzDefPointBy[translation= from A to B](y5)\tkzGetPoint{y15}
\tkzDefPointBy[translation= from A to B](y6)\tkzGetPoint{y16}
\tkzDefPointBy[translation= from A to B](y7)\tkzGetPoint{y17}
\tkzDefPointBy[translation= from A to B](y8)\tkzGetPoint{y18}
\tkzDefPointBy[translation= from A to B](y9)\tkzGetPoint{y19}
\tkzDefPointBy[translation= from A to B](D)\tkzGetPoint{C}
\tkzDefPointBy[translation= from A to D](x1)\tkzGetPoint{x11}
\tkzDefPointBy[translation= from A to D](x2)\tkzGetPoint{x12}
\tkzDefPointBy[translation= from A to D](x3)\tkzGetPoint{x13}
\tkzDefPointBy[translation= from A to D](x4)\tkzGetPoint{x14}
\tkzDefPointBy[translation= from A to D](x5)\tkzGetPoint{x15}
\tkzDefPointBy[translation= from A to D](x6)\tkzGetPoint{x16}
\tkzDefPointBy[translation= from A to D](x7)\tkzGetPoint{x17}
\tkzDefPointBy[translation= from A to D](x8)\tkzGetPoint{x18}
\tkzDefPointBy[translation= from A to D](x9)\tkzGetPoint{x19}

\tkzInterLL(x5,x15)(y5,y15)\tkzGetPoint{z}
\tkzInterLL(x1,x11)(y1,y11)\tkzGetPoint{z11}

\tkzDefPointBy[translation= from A to z11](y1)\tkzGetPoint{z12}
\tkzDefPointBy[translation= from A to z11](z12)\tkzGetPoint{z23}
\tkzDefPointBy[translation= from A to z11](z23)\tkzGetPoint{z34}
\tkzDefPointBy[translation= from A to z11](z34)\tkzGetPoint{z45}
\tkzDefPointBy[translation= from A to z11](z45)\tkzGetPoint{z56}
\tkzDefPointBy[translation= from A to z11](z56)\tkzGetPoint{z67}
\tkzDefPointBy[translation= from A to z11](z67)\tkzGetPoint{z78}
\tkzDefPointBy[translation= from A to z11](z78)\tkzGetPoint{z89}
\tkzDefPointBy[translation= from A to z11](x1)\tkzGetPoint{z21}
\tkzDefPointBy[translation= from A to z11](z21)\tkzGetPoint{z32}
\tkzDefPointBy[translation= from A to z11](z32)\tkzGetPoint{z43}
\tkzDefPointBy[translation= from A to z11](z43)\tkzGetPoint{z54}
\tkzDefPointBy[translation= from A to z11](z54)\tkzGetPoint{z65}
\tkzDefPointBy[translation= from A to z11](z65)\tkzGetPoint{z76}
\tkzDefPointBy[translation= from A to z11](z76)\tkzGetPoint{z87}
\tkzDefPointBy[translation= from A to z11](z87)\tkzGetPoint{z98}
\tkzDefPointBy[translation= from A to z11](z11)\tkzGetPoint{z22}
\tkzDefPointBy[translation= from A to z11](z22)\tkzGetPoint{z33}
\tkzDefPointBy[translation= from A to z11](z33)\tkzGetPoint{z44}
\tkzDefPointBy[translation= from A to z11](z44)\tkzGetPoint{z55}
\tkzDefPointBy[translation= from A to z11](z55)\tkzGetPoint{z66}
\tkzDefPointBy[translation= from A to z11](z66)\tkzGetPoint{z77}
\tkzDefPointBy[translation= from A to z11](z77)\tkzGetPoint{z88}
\tkzDefPointBy[translation= from A to z11](z88)\tkzGetPoint{z99}

\tkzInterLL(x2,x12)(y8,y18)\tkzGetPoint{z28}
\tkzInterLL(x7,x17)(y8,y18)\tkzGetPoint{z78}
\tkzInterLL(x2,x12)(y3,y13)\tkzGetPoint{z23}
\tkzInterLL(x7,x17)(y3,y13)\tkzGetPoint{z73}

\tkzInterLL(x1,x11)(y9,y19)\tkzGetPoint{z19}
\tkzInterLL(x2,x12)(y9,y19)\tkzGetPoint{z29}
\tkzInterLL(x3,x13)(y9,y19)\tkzGetPoint{z39}
\tkzInterLL(x4,x14)(y9,y19)\tkzGetPoint{z49}
\tkzInterLL(x5,x15)(y9,y19)\tkzGetPoint{z59}
\tkzInterLL(x6,x16)(y9,y19)\tkzGetPoint{z69}
\tkzInterLL(x7,x17)(y9,y19)\tkzGetPoint{z79}

\tkzInterLL(x9,x19)(y1,y11)\tkzGetPoint{z91}
\tkzInterLL(x9,x19)(y2,y12)\tkzGetPoint{z92}
\tkzInterLL(x9,x19)(y3,y13)\tkzGetPoint{z93}
\tkzInterLL(x9,x19)(y4,y14)\tkzGetPoint{z94}
\tkzInterLL(x9,x19)(y5,y15)\tkzGetPoint{z95}
\tkzInterLL(x9,x19)(y6,y16)\tkzGetPoint{z96}
\tkzInterLL(x9,x19)(y7,y17)\tkzGetPoint{z97}

\tkzDrawPolygon[fill=mysand](z59,z55,z45,z44,y4,y9)
\tkzDrawPolygon[fill=myblue!80!black](z59,z56,z66,z67,z77,z78,z88,z89)
\tkzDrawPolygon[fill=myblue!70](y4,z34,z33,z23,z22,z12,z11,y1)

\tkzDrawLine[add=0 and 0.1,-Stealth](A, x9)
\tkzDrawLine[add=0 and 0.1,-Stealth](A, y9)

\tkzDrawSegment[line width=0.5pt](y1, z91)
\tkzDrawSegment[line width=0.5pt](y2, z92)
\tkzDrawSegment[line width=0.5pt](y3, z93)
\tkzDrawSegment(y4, z94)
\tkzDrawSegment(y5, z95)
\tkzDrawSegment[line width=0.5pt](y6, z96)
\tkzDrawSegment[line width=0.5pt](y7, z97)
\tkzDrawSegment[line width=0.5pt](y8, z98)
\tkzDrawSegment(y9, z99)

\tkzDrawSegment[line width=0.5pt](x1, z19)
\tkzDrawSegment[line width=0.5pt](x2, z29)
\tkzDrawSegment[line width=0.5pt](x3, z39)
\tkzDrawSegment(x4, z49)
\tkzDrawSegment(x5, z59)
\tkzDrawSegment[line width=0.5pt](x6, z69)
\tkzDrawSegment[line width=0.5pt](x7, z79)
\tkzDrawSegment[line width=0.5pt](x8, z89)
\tkzDrawSegment(x9, z99)

\tkzDrawLine[add=-0.2 and -0.2](x1,y1)
\tkzDrawLine[add=-0.2 and -0.2](A,z11)

\tkzDrawLine[add=-0.2 and -0.2](z12,z21)
\tkzDrawLine[add=-0.2 and -0.2](z11,z22)

\tkzDrawLine[add=-0.2 and -0.2](z23,z32)
\tkzDrawLine[add=-0.2 and -0.2](z22,z33)

\tkzDrawLine[add=-0.2 and -0.2](z34,z43)
\tkzDrawLine[add=-0.2 and -0.2](z33,z44)

\tkzDrawLine[add=-0.2 and -0.2](z45,z54)
\tkzDrawLine[add=-0.2 and -0.2](z44,z55)

\tkzDrawLine[add=-0.2 and -0.2](z56,z65)
\tkzDrawLine[add=-0.2 and -0.2](z55,z66)

\tkzDrawLine[add=-0.2 and -0.2](z67,z76)
\tkzDrawLine[add=-0.2 and -0.2](z66,z77)

\tkzDrawLine[add=-0.2 and -0.2](z78,z87)
\tkzDrawLine[add=-0.2 and -0.2](z77,z88)

\tkzDrawLine[add=-0.2 and -0.2](z89,z98)
\tkzDrawLine[add=-0.2 and -0.2](z88,z99)

\tkzLabelLine[pos=.97,right=-1.0mm](A,D){$j$}
\tkzLabelLine[pos=0.85,left](A,D){$2l\!+\!1$}
\tkzLabelLine[pos=0.55,left](A,D){$l\!+\!2$}
\tkzLabelLine[pos=0.45,left](A,D){$l\!+\!1$}
\tkzLabelLine[pos=0.35,left](A,D){$l$}
\tkzLabelLine[pos=0.05,left](A,D){$1$}

\tkzLabelLine[pos=.96,above](A,B){$i$}
\tkzLabelLine[below=-2mm](x5,x6){$l\!+\!2$}
\tkzLabelLine[below=-2mm](x4,x5){$l\!+\!1$}
\tkzLabelLine[below=0mm](x3,x4){$l$}
\tkzLabelLine[pos=0.05,below](A,B){$1$}
\tkzLabelLine[pos=0.95,below=-3mm](A,x9){$2l\!+\!1$}

\tkzLabelPoint[above left](z27){\Huge$\tilde{\textbf{G}}$}
\tkzLabelPoint[shift={(-6.3mm,5.8mm)}](z78){\Huge$\tilde{\textbf{E}}$}
\tkzLabelPoint[shift={(-6.3mm,5.8mm)}](z23){\Huge$\tilde{\textbf{F}}$}

\end{tikzpicture}
}
\caption{Odd case. ${\tilde{\text{E}}}+\tilde{\text{G}}+\tilde{\text{F}}\leqslant{\textstyle\frac{{l}({l}-1)}{2}}\,{+}\,{\textstyle\frac{{l}({l}+3)}{2}}\,{+}\,{\textstyle\frac{{l}({l}-1)}{2}}$}
\label{fig:odd:3}
\end{figure}

Next, using (\ref{Lemma.4.odd''}), we get  (see Figure~\ref{fig:odd:3})
\begin{align*}
|{E}^{<}|\!\ =\ &|\{\la{i},{j}\ra\in{E}^{<},\:{i}>{l}\!+\!1\}|\ +\\
+\ &|\{\la{i},{j}\ra\in{E}^{<},\:{i}\leq{l}\!+\!1,\:{j}>{l}\}|\ +\\
+\ &|\{\la{i},{j}\ra\in{E}^{<},\:{i}\leq{l}\!+\!1,\:{j}\leq{l}\}|\ \leq\\
\leq\ &{\textstyle\frac{{l}({l}-1)}{2}}\ + \ {\textstyle\frac{{l}({l}+3)}{2}}\ +\ {\textstyle\frac{{l}({l}-1)}{2}}\ =\ 
{\textstyle\frac{{l}(3{l}+1)}{2}}.
\end{align*}
It follows that
$$
\frac{|{E}^{<}|}{|{E}|}\ \leq\  \frac{\frac{{l}(3{l}+1)}{2}}{\frac{2{l}(2{l}+1)}{2}}\ =\ 
\frac{3{l}+1}{4{l}+2} 
\ <\  \frac{3}{4}.
$$
\end{proof}

\begin{prop}\label{prop.grater}
$\mathsf{EMN}(\mathsf{sCop})\geq\frac{3}{4}$.
\end{prop}

As a corollary, we have

\begin{theo}\label{teor.EN(Cop,sCop)=3/4}
$\mathsf{EMN}(\mathsf{Cop})=\mathsf{EMN}(\mathsf{sCop})=\frac{3}{4}$. \qed
\end{theo}

\begin{proof}[Proof of Proposition~\ref{prop.grater}.] 
For a natural ${l}$, consider the following tournament $\ST_{l}=\la\SV_{l},\SE_{l}\ra$, where
\[\SV_{l}=\{1,\ldots,2{l}+1\}\quad\text{and}\]
\[\SE_{l}=\big\{\la{i},{j}\ra\in\SV_{l}{\times}\SV_{l}:{j}\in\{{i}\oplus1,\ldots,{i}\oplus{l}\}\big\},\]
where ${\oplus}$ is the addition in the ring $\mathbb{Z}/(2{l}{+}1)\mathbb{Z}=\{1,\ldots,2{l}+1\}$ of positive integers modulo $2{l}+1$.

For the sake of readability, denote the ordered pair $\la{m},{i}\ra$ by ${m}\!{\medvert}\!{i}$. 
For a natural ${l}$, consider the following tournament $\mathsf{T}_{l}=\la\mathsf{V}_{l},\mathsf{E}_{l}\ra$, where
\[
\mathsf{V}_{l}=\big\{{m}\!{\medvert}\!{i}:{m},{i}\in\{1,\ldots,2{l}{+}1\}\big\}\quad\text{and}
\]
\begin{align*}
\mathsf{E}_{l}\:=\ &\big\{\la{m}\!{\medvert}\!{i},{n}\!{\medvert}\!{j}\ra\in\mathsf{V}_{l}^2:
{m}>{n},\:{i}={j}\big\}\ \cup\\
\cup\ &\big\{\la{m}\!{\medvert}\!{i},{n}\!{\medvert}\!{j}\ra\in\mathsf{V}_{l}^2:
{m}={n},\:\la{i},{j}\ra\in\SE_{l}\big\}\ \cup\\
\cup\ &\big\{\la{m}\!{\medvert}\!{i},{n}\!{\medvert}\!{j}\ra\in\mathsf{V}_{l}^2:{m}\neq{n},\:{i}\neq{j},\:\la{m},{n}\ra\in\SE_{l}\big\}.
\end{align*}

For every natural ${l}$ and ${m},{i}\in\{1,\ldots,2{l}{+}1\}$, we have
\begin{align*}
({m}\!{\medvert}\!{i})^{+}_{\mathsf{T}_{l}}\ =
\ &\big\{\la{m}\!{\medvert}\!{i},{n}\!{\medvert}\!{i}\ra\in\mathsf{V}_{l}^2:
{m}>{n}\big\}\ \cup\\
\cup\ &\big\{\la{m}\!{\medvert}\!{i},{m}\!{\medvert}\!{j}\ra\in\mathsf{V}_{l}^2:
\la{i},{j}\ra\in\SE_{l}\big\}\ \cup\\
\cup\ &\big\{\la{m}\!{\medvert}\!{i},{n}\!{\medvert}\!{j}\ra\in\mathsf{V}_{l}^2:{m}\neq{n},\:{i}\neq{j},\:\la{m},{n}\ra\in\SE_{l}\big\},
\end{align*}
so that
\begin{align*}
\d({m}\!{\medvert}\!{i},\mathsf{T}_{l})\ =\ &({m}-1)\ +\ {l}\ +\  
\sum_{{j}\neq{i}}\big|\big\{\la{m}\!{\medvert}\!{i},{n}\!{\medvert}\!{j}\ra\in\mathsf{V}_{l}^2:{m}\neq{n},
\:\la{m},{n}\ra\in\SE_{l}\big\}\big|
\ =\\
=\ &({m}-1)\ +\ {l}\ +\ (2{l})\cdot{l}.
\end{align*}
Therefore, for every strict Copeland fair ranking ${r}$ of $\mathsf{T}_{l}$, we have
\[
{m}<{n}\quad\text{implies}\quad{r}({m}\!{\medvert}\!{i})<{r}({n}\!{\medvert}\!{j})\quad\text{for all }{m}\!{\medvert}\!{i},{n}\!{\medvert}\!{j}\in\mathsf{V}_{l}.
\]

It follows that
\[
\big\{\la{m}\!{\medvert}\!{i},{n}\!{\medvert}\!{j}\ra\in\mathsf{V}_{l}^2:{m}<{n},\:{i}\neq{j},\:\la{m},{n}\ra\in\SE_{l}\big\}
\ \subseteq\ \overleftarrow{\mathsf{E}}({r},\mathsf{T}_{l}),
\]
therefore
\[
\begin{split}
|\overleftarrow{\mathsf{E}}({r},\mathsf{T}_{l})|\ &\geq\ 
\big|\big\{\la{m}\!{\medvert}\!{i},{n}\!{\medvert}\!{j}\ra\in\mathsf{V}_{l}^2:{m}<{n},\:{i}\neq{j},\:\la{m},{n}\ra\in\SE_{l}\big\}\big|
\ =\\ 
&=\ \sum_{{i}=1}^{2{l}+1}\big|\big\{\la{m}\!{\medvert}\!{i},{n}\!{\medvert}\!{j}\ra\in\mathsf{V}_{l}^2:{m}<{n},\:{i}\neq{j},\:\la{m},{n}\ra\in\SE_{l}\big\}\big|\ =\\ 
&=\ (2{l}+1)\big|\big\{\la{m}\!{\medvert}\!{1},{n}\!{\medvert}\!{j}\ra\in\mathsf{V}_{l}^2:{m}<{n},\:{1}\neq{j},\:\la{m},{n}\ra\in\SE_{l}\big\}\big|
\ =\\ 
&=\ \sum_{{j}=2}^{2{l}+1}(2{l}+1)\big|\big\{\la{m}\!{\medvert}\!{1},{n}\!{\medvert}\!{j}\ra\in\mathsf{V}_{l}^2:{m}<{n},\:\la{m},{n}\ra\in\SE_{l}\big\}\big|
\ =\\ 
&=\ 2{l}(2{l}+1)\big|\big\{\la{m}\!{\medvert}\!{1},{n}\!{\medvert}\!{2}\ra\in\mathsf{V}_{l}^2:{m}<{n},\:\la{m},{n}\ra\in\SE_{l}\big\}\big|
\ =\\ 
&=\ 2{l}(2{l}+1)\sum_{{m}=1}^{2{l}+1}\big|\big\{\la{m}\!{\medvert}\!{1},{n}\!{\medvert}\!{2}\ra\in\mathsf{V}_{l}^2:{m}<{n},\:{n}\in\{{m}\oplus1,\ldots,{m}\oplus{l}\}\big\}\big|
\ =\\ 
&=\ 2{l}(2{l}+1)\big(\underbrace{{l}+\ldots+{l}}_{{l}+1}\,+\,({l}{-}1)+\ldots+0\big)
\ =\ {l}^2(2{l}{+}1)(3{l}{+}1).
\end{split}
\]
Since 
\[
|\mathsf{E}_{l}|=\frac{((2{l}{+}1)^2-1)(2{l}{+}1)^2}{2}=2{l}({l}{+}1)(2{l}{+}1)^2,
\]
we have
\[
\min\limits_{{r}\in\mathsf{sCop}(\mathsf{T}_{l})}
\frac{|\overleftarrow{\mathsf{E}}({r},\mathsf{T}_{l})|}{|\mathsf{E}_{l}|}\geq
\frac{{l}^2(2{l}{+}1)(3{l}{+}1)}{2{l}({l}{+}1)(2{l}{+}1)^2}.
\]
It follows that 
\[
\mathsf{EMN}(\mathsf{sCop})=\sup\limits_{{T}\in\mathbb{T}}
\min\limits_{{r}\in\mathsf{sCop}({T})}\frac{|\overleftarrow{\mathsf{E}}({r},{T})|}{|\mathsf{E}({T})|}
\geq\lim\limits_{{l}\to\infty}\frac{{l}^2(2{l}{+}1)(3{l}{+}1)}{2{l}({l}{+}1)(2{l}{+}1)^2}=\frac{3}{4}.
\]
\end{proof}


\printbibliography

\end{document}